\documentclass[12pt]{article}
\usepackage{amsmath}
\usepackage{amsthm}
\usepackage{amssymb,amsfonts}
\usepackage{fancyhdr}
\usepackage{hyperref}
\usepackage{latexsym,microtype}
\usepackage{todonotes}
\usepackage{nicefrac}
\usepackage{epsfig,euscript}
\usepackage{stmaryrd}
\usepackage{setspace,mathrsfs}
\usepackage{enumerate}
\usepackage[all]{xypic}
\usepackage{bbm,ifpdf,tikz}
\ifpdf
\usepackage{pdfsync}
\fi

\newtheorem{theorem}{Theorem}[section]

\newtheorem{lemma}[theorem]{Lemma}
\newtheorem{proposition}[theorem]{Proposition}
\newtheorem{corollary}[theorem]{Corollary}

\newtheorem{conjecture}[theorem]{Conjecture}
\newtheorem{question}[theorem]{Question}

{
\theoremstyle{definition}
\newtheorem{definition}[theorem]{Definition}
\newtheorem{example}[theorem]{Example}

\newtheorem{remark}[theorem]{Remark}
}

\newenvironment{proofComputationQreg}{\noindent {\bf Proof of Theorem \ref{ComputationQreg}:}}{\qed \par}

\newcommand{\excise}[1]{}
\newcommand{\Spec}{\operatorname{Spec}}

\newcommand{\Proj}{\operatorname{Proj}}

\newcommand{\supp}{\operatorname{Supp}}

\renewcommand{\dim}{\operatorname{dim}}

\newcommand{\Sym}{\operatorname{Sym}}
\newcommand{\rk}{\operatorname{rk}}

\renewcommand{\and}{\qquad\text{and}\qquad}

\newcommand{\Hom}{\operatorname{Hom}}

\newcommand{\Z}{\mathbb{Z}}
\newcommand{\Q}{\mathbb{Q}}
\newcommand{\N}{\mathbb{N}}
\newcommand{\R}{\mathbb{R}}
\newcommand{\C}{\mathbb{C}}

\renewcommand{\H}{\operatorname{H}}

\renewcommand{\a}{\alpha}

\newcommand{\la}{\lambda}

\newcommand{\scrA}{\mathscr{A}}

\newcommand{\cs}{\mathbb{G}_m}

\newcommand{\Lie}{\operatorname{Lie}}
\newcommand{\fg}{\mathfrak{g}}

\newcommand{\fk}{\mathfrak{k}}
\newcommand{\ft}{\mathfrak{t}}

\newcommand{\cO}{\mathcal{O}}

\renewcommand{\sl}{\mathfrak{sl}}

\newcommand{\HH}{H\! H}

\renewcommand{\H}{H}
\newcommand{\omreg}{\omega^{\operatorname{reg}}}
\newcommand{\Xreg}{X^{\operatorname{reg}}}
\newcommand{\tX}{\tilde{X}}
\newcommand{\tXreg}{\tX^{\operatorname{sm}}}

\newcommand{\scrX}{\mathscr{X}}

\newcommand{\reg}{{\mathrm{reg}}}
\newcommand{\Treg}{T_\reg}
\newcommand{\Mreg}{M_\reg}
\newcommand{\Rreg}{R_\reg}

\newcommand{\cohtor}{K}
\newcommand{\cohtorreg}{K_{\reg}}
\newcommand{\Sreg}{S_\reg}

\newcommand{\Freg}{F_\reg}
\newcommand{\Ireg}{I_\reg}
\newcommand{\Ereg}{E_\reg}

\newcommand{\Qreg}{Q_\reg}
\newcommand{\Dreg}{D_\reg}
\newcommand{\barEreg}{\bar{E}_\reg}
\newcommand{\barIreg}{\bar{I}_\reg}
\newcommand{\barQreg}{\bar{Q}_\reg}
\newcommand{\Nreg}{N_\reg}
\newcommand{\Aut}{\operatorname{Aut}}

\newcommand{\tr}{\operatorname{tr}}
\newcommand{\Azo}{\scrA_0^2}

\newcommand{\Uhg}{U_{\hspace{-1pt}\hbar}\hspace{1pt}\fg}
\newcommand{\xiHC}{\varphi}
\newcommand{\an}{\operatorname{an}}

\newcommand{\Dd}{D}
\newcommand{\cmt}{t}
\newcommand{\AKZ}{}
\newcommand{\AKZsup}{AKZ}

\newcommand{\cohpsi}{\widetilde{\psi}}
\newcommand{\qconnsign}{-}
\newcommand{\qconnosign}{+}
\newcommand{\GHC}{\Phi}
\newcommand{\pointrest}{r}

\newcommand{\nicktodo}{\todo[inline,color=green!20]}
\newcommand{\michtodo}{\todo[inline,color=blue!20]}

\begin{document}
\noindent{\Large\bf The quantum Hikita conjecture}\\

\noindent{\bf Joel Kamnitzer}\\
Department of Mathematics, University of Toronto, Toronto, Ontario, M5S 2E4\\

\noindent{\bf Michael McBreen}\\
Department of Mathematics, University of Toronto, Toronto, Ontario, M5S 2E4\\

\noindent{\bf Nicholas Proudfoot}\\
Department of Mathematics, University of Oregon,
Eugene, OR 97403\\

{\small
\begin{quote}
\noindent {\em Abstract.}
The Hikita conjecture relates the coordinate ring of a conical symplectic singularity
to the cohomology ring of a symplectic resolution of the dual conical symplectic singularity.
We formulate a quantum version of this conjecture, which relates the quantized coordinate
ring of the first variety to the quantum cohomology of a symplectic resolution of the dual variety.
We prove this conjecture for hypertoric varieties and for the Springer resolution. 

Our paper includes an appendix, written by Ben Webster, which studies highest weights for quantizations of symplectic resolutions with isolated torus actions.
\end{quote} }

\spacing{1.2}
\section{Introduction}
A fascinating phenomenon in the theory of conical symplectic resolutions is that they tend to come in dual pairs.
The exact definition of this notion of ``symplectic duality" is somewhat in flux; a proposed definition in terms of a certain category $\cO$
was formulated by the third author and collaborators in \cite[Section 10]{BLPWgco}, but it is not clear that this definition is flexible enough to encompass
all of the examples that one wants to consider.  Nonetheless, there is broad agreement on certain basic families of examples:
the Springer resolution is dual to the Springer resolution for the Langlands dual group,
\cite[Theorem 1.1.3]{BGS96}, hypertoric varieties are dual to other hypertoric varieties \cite[Theorem 1.2]{BLPWtorico}, affine type A quiver varieties are dual to other such varieties \cite[Corollary 5.25]{Webqui}.  Finite ADE quiver varieties
are dual to slices in the affine Grassmannian for the Langlands dual group \cite[Remark 10.7]{BLPWgco} and \cite{KTWWY2}.
Finally, and perhaps most important, given a linear representation of a reductive group, the Higgs branch of the associated 3-dimensional
$N=4$ supersymmetric gauge theory (defined as a hyperk\"ahler quotient) is dual to the Coulomb branch
of the same theory (defined in \cite{BFN}), at least when the two spaces are sufficiently well behaved \cite{WCoulomb}.

Of the various manifestations of symplectic duality in terms of algebraic invariants of the resolutions, one of the most attractive is due
to Hikita \cite{Hikita}.  Let $\tX\to X$ and $\tX^!\to X^!$ be a dual pair of conical symplectic resolutions, and let $T$ be a maximal torus
in the Hamiltonian automorphism group of $\tX$.  Hikita observed that, for many of the aforementioned examples,
the coordinate ring of the fixed scheme $X^T$ is isomorphic to the cohomology ring of $\tX^!$.
Specifically, he proved this for hypertoric varieties, finite type A quiver varieties, and the Hilbert scheme of points
in the plane (which is self-dual), and he asked whether this phenomenon might hold for other examples of symplectic duality.  For affine Grassmannian slices, this was proved by the first author and collaborators \cite[Theorem 8.1]{KTWWY}.
We will refer to this isomorphism of algebras as the {\bf Hikita conjecture}.

The Hikita conjecture was extended by Nakajima \cite[Conjecture 8.9]{KTWWY}, who proposed that the equivariant
cohomology of $\tX^!$ for the conical $\cs$-action should coincide with the $B$-algebra of the quantized coordinate
ring of $X$, with the equivariant parameter for the conical action identified with the quantization parameter for the coordinate ring.
The $B$-algebra is an object that was introduced in \cite[Section 5.2]{BLPWgco}
to construct the standard and costandard objects of category $\cO$
\cite[Section 5.2]{BLPWgco}.  We will refer to Nakajima's extension as the {\bf equivariant Hikita conjecture}.  In \cite[Theorem 1.5]{KTWWY2}, the first author and collaborators established a weak form of the equivariant Hikita conjecture for affine Grassmannian slices.

Our goal is to introduce yet another level of complexity to the Hikita conjecture.  On one side of our conjecture,
we will have the {\bf specialized quantum D-module} of $\tX^!$.  As a vector space, this is basically the equivariant quantum cohomology ring
(see Remark \ref{not quite the ring}),
but it is equipped with the structure of a module over the Rees algebra of a certain ring of differential operators, where the module structure
is related to quantum multiplication by divisors.  The beautiful structures attached to the quantum D-module of a conical symplectic resolution have been the subject of much recent interest, starting with \cite{okounkov2010quantum} and \cite{BMO}. For a sample of subsequent works, see \cite{QGQC, okounkov2016quantum, McBS, McBP, aganagic2018quantum, okounkov2015enumerative}.
The word ``specialized" refers to the fact that we identify the Rees parameter with the equivariant parameter for the conical action,
which is a major simplification (see Remark \ref{CY}).

On the other side of our conjecture, we have an object that serves as the universal
source for graded traces of representations of the quantized coordinate ring of $X$, just as degree zero Hochschild homology serves as the universal
source for ordinary traces.  More specifically,
let $\scrA$ be the canonical quantization of the
universal filtered Poisson deformation of $\tX$.  Thus $\scrA$ is a non-commutative algebra with a large centre, and the various central
quotients of $\scrA$ are each quantizations of the coordinate ring of $X$.  For example, if $X$ is the nilpotent cone in a reductive Lie algebra,
then $\scrA$ is a finite extension of the corresponding universal enveloping algebra.
The algebra $\scrA$ comes with two compatible gradings, one into weight spaces for the Hamiltonian torus action, and an additional
$\N$-grading into weight spaces for the conical $\cs$-action.  Let $\scrA_0^2$ be the part of $\scrA$ that lies
in weight 0 for the Hamiltonian torus action and degree 2 for the conical action.  Let $S$ be the algebra with basis elements $q^{\la}$,
where $\la$ is a element of the semigroup generated by certain weights of the Hamiltonian torus action called {\bf equivariant roots}
(Section \ref{sec:eq-roots}).  We then define $M$ to be the quotient of $S\otimes\Sym\scrA_0^2$
by the $S$-linear span of elements of the form
$1\otimes ab - q^\la \otimes ba$, where $a,b\in\scrA$ are elements of weight $\la$ and $-\la$, respectively.  This vector space $M$
is not a ring, but rather a module under the action of the Rees algebra of a certain ring of differential operators (Proposition \ref{Rsub}).

The definition of $M$ is motivated as follows.  Let $V$ be a graded module over $\scrA$, and for any weight $\mu$ of the Hamiltonian torus,
let $V_\mu \subset V$ be the corresponding weight space (see Section \ref{sec:traces} for a more precise definition).  If $V$ is suitably
well behaved, then we have a {\bf graded trace} map that takes an element $a\in\scrA_0$ to a power series where the coefficient of $q^\mu$
is equal to the trace of $a$ on $V_\mu$.  We then prove that the graded trace map factors through $M$ (Proposition \ref{lemma characters}).
For this reason, we call $M$ the {\bf D-module of graded traces}.
We note that after specializing $q=\hbar=1$, $M$ coincides with the degree zero Hochschild homology of $\scrA$ 
(Proposition \ref{M1}), and we obtain the ordinary
trace map for a finite dimensional representation of $\scrA$.

Our main conjecture (Conjecture \ref{pi}) says that, after inverting some parameters associated with the equivariant roots, $M$ can be identified with the specialized quantum D-module of $\tX^!$, thus relating the quantization of $X$ to the quantum cohomology of $\tX^!$. We call this the {\bf quantum Hikita conjecture}.

\begin{conjecture} If $X$ and $X^!$ are dual conical symplectic singularities and $\tX^!$ is a symplectic
resolution of $X^!$, then the D-module of graded traces for $X$
may be identified with the specialized quantum D-module of $\tX^!$ away from the root hyperplanes.
\end{conjecture}

\begin{theorem}The quantum Hikita conjecture holds for hypertoric varieties and for Springer resolutions
{\em (Theorems \ref{hypertoric-pi} and \ref{Springer-pi})}.
\end{theorem}

In addition to being interesting in its own right, the quantum Hikita conjecture relates to various previous conjectures by specializing $q$.
If we set $q$ equal to zero, then $M$ turns into $B$-algebra of $\scrA$ (Proposition \ref{M0} and
Remark \ref{Nak-abelian}), and our conjecture specializes to a version of the equivariant Hikita conjecture (Remark \ref{Nak-Hik}).
On the other hand, setting $q$ equal to 1, $M$ turns into the degree zero Hochschild homology of $\scrA$,
which is conjecturally related to the intersection cohomology of $X^!$ \cite[Conjecture 3.6]{Pro12}.  Similarly, the quantum cohomology
of $\tX^!$ at $q=1$ is also conjecturally related to the intersection cohomology of $X^!$ \cite[Conjecture 2.5]{McBP}.
Thus our conjecture provides a bridge between these two previous conjectures of the second and third authors (Remark \ref{whoa}).

\begin{remark}One of the original motivations for this work was the case where $X$ (respectively $X^!$) is the Coulomb 
(respectively Higgs) branch of a
3-dimensional $N=4$ supersymmetric gauge theory. In this case, there is a clear heuristic for the relation between the
module of graded traces for $X$ and the specialized quantum D-module of $\tX^!$.
Indeed, in this case the specialized quantum D-module is encoded in the differential relations satisfied by a certain function
called the $I$-function. The $I$-function is a generating function for equivariant volumes of moduli spaces of quasimaps
from a rational curve into $\tX^!$. These quasimaps, in turn, are closely related to the moduli spaces used to define $X$ in \cite{BFN}.  
In a forthcoming paper, the first author together with Justin Hilburn and Alex Weekes will prove that the quantization of $ X $ acts on the homology of certain quasimap moduli spaces for $ \tX^!$. 
\end{remark}

\begin{remark}A different possible line of investigation is to replace the equivariant cohomology of $\tX^!$ by its equivariant $K$-theory. Then the specialized quantum D-module must be replaced by a module over difference operators, which has in many respects proved to be an even richer object \cite{okounkov2016quantum, aganagic2018quantum, okounkov2015enumerative}. It would be interesting to see how our conjecture adapts to this setting.
\end{remark}

\vspace{\baselineskip}
\noindent
{\em Acknowledgments:}
The authors would like to thank Alexander Braverman for proposing the problem of formulating a quantum version of the Hikita conjecture and for many helpful discussions. 
The authors are also grateful to Roman Bezrukavnikov, Pavel Etingof, Davide Gaiotto, Sam Gunningham, Ivan Losev, Davesh Maulik, Hiraku Nakajima, Andrei Negut, Andrei Okounkov, Peng Shan and Ben Webster for stimulating conversations.  JK was supported by an NSERC discovery grant.
MM completed part of this work at the Massachusetts Institute of Technology, 
the Hausdorff Center for Mathematics, and the Yau Mathematical Sciences Center.
NP was supported by NSF grant DMS-1565036 and would like to thank le Ch\^ateau de Trintange for its hospitality
during the last stages of the completion of this manuscript.

\section{Conical symplectic singularities}\label{sec:cones}
Let $X$ be a {\bf conical symplectic singularity of weight two}.
By this we mean that
$X$ is a normal affine Poisson variety over $\C$ equipped with an action of $\cs$ satisfying the following conditions:
\begin{itemize}
\item the coordinate ring $\mathcal{O}(X)$ is non-negatively graded by the action of $\cs$, with the degree zero part
consisting only of constant functions and the degree one part being zero\footnote{This last condition rules out
the degenerate example $X = \C^2$, or anything with a factor of $\C^2$.}
\excise{
\nicktodo{I'm not sure exactly where to make this observation, but here is why I want to assume that $\mathcal{O}(X)^1 = 0$.
In Section \ref{sec:duality}, we say that we want an isomorphism $\mathcal{O}(X)_0^2 \cong \ft \cong H^2(\tX^!; \C)$.
But the basic Hikita conjecture says that $H^2(\tX^!; \C)$ is isomorphic to the quotient of $\mathcal{O}(X)_0^2$
by products of elements of degree 1 with opposite nonzero weights.  I don't want there to be any such products,
and there won't be if there are no nonzero functions in degree 1.}
}
\item the Poisson bracket has degree -2 with respect to this grading
\item the Poisson bracket is induced by a symplectic form $\omreg$ on the smooth locus $\Xreg$
\item for some (equivalently any) projective resolution $\pi:\tX\to X$, the 2-form $\pi^*\omreg$ extends
to a (possibly degenerate) 2-form on $\tX$.
\end{itemize}
Examples include the nilpotent cone of a simple Lie algebra, hypertoric varieties, quiver varieties, and certain
subvarieties of the affine Grassmannian.

\subsection{The Hamiltonian automorphism group}
Let $\mathcal{O}(X)^2$ be the degree 2 part of $\mathcal{O}(X)$.  Since the Poisson bracket on $\mathcal{O}(X)$ has degree -2,
$\mathcal{O}(X)^2$ is a Lie subalgebra of $\mathcal{O}(X)$.  This Lie algebra acts by graded endomorphisms on $\mathcal{O}(X)$.
Assume that there exists a reductive group $ \Aut(X)$, whose Lie algebra is $ \mathcal{O}(X)^2 $, and which acts faithfully by Poisson automorphisms on $\mathcal{O}(X)$, integrating the action of $\mathcal{O}(X)^2$.

\begin{remark}
If $X$ admits a symplectic resolution $\tX$, then the Lie algebra $\mathcal{O}(X)^2$
may be identified with the Lie algebra of Hamiltonian vector fields on $\tX$.
For this reason, we refer to $\Aut(X)$ as the {\bf Hamiltonian automorphism group} of $X$.
If, in addition, $\tX$ admits a hyperk\"ahler metric compatible with the symplectic form, then
we expect $\Aut(X)$ to be the complexification of the group of hyperk\"ahler automorphisms of $\tX$.
This gives at least a heuristic reason to believe that the Lie algebra $\mathcal{O}(X)^2$ integrates to a reductive group.
\end{remark}

Let $T\subset \Aut(X)$ be a maximal torus, and let $\ft := \Lie(T)$ be the Lie algebra of $T$.
The action of $T$ on $X$ induces a second grading on coordinate ring $\mathcal{O}(X)$ by the group $\ft^*_\Z := \Hom(T,\cs)$.
Since the action of $T$ commutes with the action of $\cs$, this second grading is
compatible with the grading by $\N$.
For any $\la\in\ft^*_\Z$ and $k\in\N$, we let define $\mathcal{O}(X)_\la$, $\mathcal{O}(X)^k$, and $\mathcal{O}(X)_\la^k := \mathcal{O}(X)_\la\cap\mathcal{O}(X)^k$ to be the corresponding
isotypic components for the actions of $T$, $\cs$, and $T\times\cs$, respectively.
Since $\Aut(X)$ is reductive, the zero root space $\mathcal{O}(X)^2_0\subset\mathcal{O}(X)^2$ is equal to the Cartan subalgebra $\ft\subset\mathcal{O}(X)^2$.

\excise{
\begin{proposition}\label{Poisson-Lie}
We have $\ft=\mathcal{O}(X)^2_0\subset\mathcal{O}(X)^2$.
\end{proposition}

\begin{proof}
We have assumed that the Lie group $\Aut(X)$ is reductive, therefore the Cartan subalgebra coincides with the zero root space.
\end{proof}
}

\subsection{Deformation and quantization} \label{DeformQuant}
Choose a $\Q$-factorial terminalization $\tX$ of $X$, as in \cite[Proposition 2.3]{Losev-orbit}, and consider the smooth locus
$\tXreg\subset\tX$.
Let $\tilde\scrX$ be the universal filtered Poisson deformation of $\tX$, which has base $H^2(\tXreg; \C)$.  Let $\scrX := \Spec\cO(\tilde\scrX)$,
which is a deformation of $X$ over $H^2(\tXreg; \C)$.
Two different choices of $\tX$ will yield two isomorphic families $\scrX$, and the isomorphism between them is canonical up to the action of the Namikawa Weyl group \cite[Corollary 2.13]{Losev-orbit}.

Let $\scrA$ be the canonical quantization of $\scrX$.
This is an $\N$-graded algebra over the ring $\Sym H_2(\tXreg; \C)\otimes\C[\hbar]$,
with $H_2(\tXreg; \C)$ and $\hbar$ both in degree 2.  If we set $\hbar$ equal to 1, we obtain the canonical filtered
quantization of \cite[Proposition 3.3]{Losev-orbit}.  The existence of such a quantization follows
from the work of Bezrukavnikov-Kaledin and Losev; see \cite[Sections 3.1-3.3]{BLPWquant} for details.

Let $\scrA^2$ denote the degree 2 part of $\scrA$.  This is naturally a Lie algebra, with Lie bracket given by
$\hbar^{-1}$ times the commutator.  The centre of the Lie algebra $\scrA^2$ contains $H_2(\tXreg; \C)\oplus\C\hbar$,
and the quotient of $\scrA^2$ by this subalgebra is canonically isomorphic to $\mathcal{O}(X)^2$.  That is, we have an exact sequence of Lie
algebras
\begin{equation}\label{first exact}0\to H_2(\tXreg; \C)\oplus\C\hbar \to \scrA^2 \to \mathcal{O}(X)^2 \to 0,\end{equation}
where $H_2(\tXreg; \C)\oplus\C\hbar$ is endowed with the trivial Lie bracket.
For any $x\in\scrA^2$, let $\bar x$ denote the image of $x$ in $\mathcal{O}(X)^2$.

The Lie algebra $\scrA^2$ acts on $\scrA$ by $\hbar^{-1}$ times the commutator; furthermore, the central subalgebra
$H_2(\tXreg; \C)\oplus\C\hbar$ acts trivially,
so we obtain an action of the Lie algebra $\mathcal{O}(X)^2$ on $\scrA$.  Since the action of $\mathcal{O}(X)^2$ on $\mathcal{O}(\scrX)$ integrates
to an action of $\Aut(X)$ and $\scrA$ is a flat deformation of $\mathcal{O}(\scrX)$ over the affine line, the action of $\mathcal{O}(X)^2$ on $\scrA$
also integrates to an action on $\Aut(X)$.
This endows $\scrA$ with a direct sum decomposition $$\scrA = \bigoplus_{\la\in\ft^*_\Z}\scrA_\la$$ into $T$-weight spaces,
where
$$\scrA_{\la} := \{a\in\scrA\mid [x,a] = \hbar\langle\la,\bar x\rangle a\;\;\text{for all $x\in\scrA^2_0$}\}.$$
This decomposition is compatible with the grading by $\N$.
Taking zero weight spaces in the exact sequence \eqref{first exact}, we obtain an exact sequence
\begin{equation}\label{eq:quant-exact}
0\to H_2(\tXreg; \C)\oplus\C\hbar \to \Azo \to \ft \to 0,
\end{equation}
which we call the {\bf quantization exact sequence}.  This exact sequence will play a major role
in the formulation of our main conjecture.

\begin{remark} \label{rem:qex}
Since $\Azo$ is abelian, the quantization exact sequence splits.
Choosing a splitting is equivalent to choosing a quantum
comoment map for the action of $T$ on $\scrA$.
\end{remark}

There are two main examples which we will work with in this paper: hypertoric varieties and the Springer resolution.

\begin{example}\label{ht example}
Suppose that $X$ is the affine hypertoric variety obtained as a symplectic quotient of
$T^*\C^n$ by a subtorus $K\subset\cs^n$.  Then we may take $T = \cs^n/K$, and $\scrA$ is isomorphic to the {\bf hypertoric enveloping algebra}
(Section \ref{sec:hea}).
If $y_i$ is the $i^\text{th}$ coordinate function on $\C^n$, then
$\Azo$ has basis $\{\hbar,y_1\partial_1,\ldots,y_n\partial_n\}$.
The map from $\Azo/\C\hbar \cong\Lie(\cs^n)$ to $\ft$ is induced by the map of algebraic groups from $\cs^n$ to $T$.
This example will be studied in greater detail in Section \ref{sec:hypertoric}.
\end{example}

\begin{example}\label{Springer example}
Let $ G $ be a semisimple complex group and let $ X \subset \fg^*$ be the union of those coadjoint orbits that are preserved by dilations.
If we use the Killing form to identify $\fg^*$ with $\fg$, then $X$ is taken to the nilpotent cone of $\fg$, so we will refer to $X$ as the {\bf nilpotent cone}.
Then $X$ is a conical symplectic singularity of weight two, where the Poisson structure comes from restricting the usual Poisson structure on $\mathfrak g^* $ and the action of $\cs$ is by the inverse square of scalar multiplication.
Furthermore, $X$ admits a symplectic resolution $\tX = T^* (G/B) $, known as the {\bf Springer resolution}.   The group $\Aut(X)$ is 
isomorphic to the quotient of $ G $ by its centre, and $ T $ is the quotient of the maximal torus of $ G $ by the centre of $ G $.
The group $H^2(\tX; \C)$ is canonically isomorphic to $\ft^*$.
The universal Poisson deformation $ \tilde\scrX$ is isomorphic to the Grothedieck-Springer resolution $\tilde\fg^* $, and its affinization $ \scrX$ is isomorphic to $\fg^* \times_{\ft^*/W} \ft^* $.
The canonical filtered quantization is identified with the {\bf enhanced enveloping algebra} $ \mathcal{A} := U\fg \otimes_{Z(U\fg)} \Sym \ft $, and
$\scrA$ is the Rees algebra with respect to the PBW filtration.
The space $\Azo$ is generated by $\hbar$ and vectors of the form $x_1\otimes 1 + 1\otimes x_2$ for $x_1,x_2\in\ft$.
The map from $\Azo$ to $\ft$ in the quantization exact sequence takes $\hbar$ to 0 and $x_1\otimes 1 + 1\otimes x_2$ to $x_1$.
\end{example}

\section{Algebraic construction}\label{sec:algebra}
In this section we fix a conical symplectic singularity $X$ and a $\Q$-factorial terminalization $\tX$,
and we use the canonical quantization $\scrA$ from Section \ref{DeformQuant} to define the D-module of graded traces.

\subsection{Equivariant roots}\label{sec:eq-roots}
Let $\scrA^+\subset \scrA$ be the two-sided ideal spanned by classes of positive degree with respect to the $\mathbb{N}$-grading,
and consider the $T$-vector space $\scrA^+/(\scrA^+\cdot\scrA^+)$.  Let $\Sigma\subset\ft_\Z^*$ be the set of nonzero weights of
$\scrA^+/(\scrA^+\cdot\scrA^+)$.  Motivated by \cite[Definition 3.1]{okounkov2015enumerative}, we will refer to $\Sigma$ as the set of {\bf equivariant roots} of $X$.

\begin{remark}
Setting $\hbar$ equal to zero gives a canonical surjective map of $T$-representations
from $\scrA^+/(\scrA^+\cdot\scrA^+)$ to the Zariski cotangent space to $\scrX$
at the unique $(T\times\cs)$-fixed point; we expect this map to be an isomorphism.  This in turn maps
to the Zariski cotangent space to $X$, inducing a bijection on nonzero weights.
Okounkov defines the equivariant roots by choosing a symplectic resolution (if it exists) and taking the union of the nonzero
weights in the cotangent spaces of {\em all} of the $T$-fixed points of the resolution.\footnote{Okounkov uses tangent
spaces rather than cotangent spaces, but the weights are the same, since the action of $T$ preserves the symplectic form.}
We expect that our definition will coincide with Okounkov's when a symplectic resolution exists.
\end{remark}

Fix a cocharacter $\xi\in\ft_\Z$ such that $\langle \la,\xi\rangle \neq 0$ for all $\la\in \Sigma$, and let
$$\Sigma_+ := \{\la\in\Sigma\mid \langle \la,\xi\rangle > 0\}.$$
We will call elements of $\Sigma_+$ {\bf positive equivariant roots}.
Let $$\scrA_+ := \bigoplus_{\langle\la,\xi\rangle>0}\scrA_\la \and \scrA_- := \bigoplus_{\langle\la,\xi\rangle<0}\scrA_\la.$$
The following lemma says that the right ideal of $\scrA$ generated by $\scrA_+$ is in fact generated by the elements
of $\scrA_\la$ for $\la\in\Sigma_+$.

\begin{lemma} \label{putrootinfront}
If $ a \in \scrA^m_+$, 
then there exist positive equivariant roots $\la_1,\ldots,\la_n$ (possibly not distinct),
along with elements $y_i\in\scrA_{\la_i}$
and $z_i\in\scrA$ for all $i$, such that
$$a = \sum_{i=1}^n y_i z_i.$$
\end{lemma}

\begin{proof}
We proceed by induction on $m$.  We may assume that $a\in\scrA_\mu^m$ for some
$\mu\in\ft^*_\Z$ with $\langle\mu,\xi\rangle>0$, and we may assume
that the statement holds for all elements of $\scrA_+^k$ when $k<m$.

If $a\notin\scrA^+\cdot\scrA^+$, then $a$ represents a nonzero element of $ \scrA^+ / (\scrA^+\cdot\scrA^+)$, in which case $\mu\in\Sigma_+$
and we are done.  Thus we may assume that $ a \in \scrA^+\cdot\scrA^+$.  This means that we can write
$a = \sum_j b_jc_j,$ where
$b_j\in \scrA_{\mu_j}^{p_j}$ and $c_j \in \scrA_{ \mu - \mu_j}^{m-p_j}$ for some elements $\mu_j\in\ft^*_\Z$
and $p_j\in\N$ with $0<p_j<m$ for all $j$.
If $\langle\mu_j,\xi\rangle>0$, then we may apply our inductive hypothesis to $b_j$, and thus write
$b_j c_j$ in the desired form.
Alternatively, if $\langle\mu_j,\xi\rangle\leq 0$, then $ \langle\mu - \mu_j,\xi\rangle\geq\langle\mu,\xi\rangle >  0 $,
so we may apply our inductive hypothesis to $c_j$.  Finally, we note that
$$b_j c_j = c_j b_j + [b_j, c_j] = c_j b_j + \hbar d_j$$ for some $d_j \in \scrA_{\mu}^{m-2}$.
Applying our inductive hypothesis to both $c_j$ and $d_j$, we may again write $b_j c_j$ in the desired form.
\end{proof}

\subsection{The ring \boldmath{$R$}}\label{sec:ring}
Let $$S := \C\{q^\lambda \mid \lambda \in \N\Sigma_+\}
\subset \C\{q^\lambda \mid \lambda \in \ft^*_\Z\}\cong \mathcal{O}(T).$$
The fact that $T$ acts effectively on $X$ implies that $\Sigma_+$ spans $\ft^*_\Z$, and therefore that $\Spec S$ is a
(possibly non-normal) affine $T$-toric variety with a unique fixed point $0\in\Spec S$.

\excise{
\begin{remark}
There are three closely related rings that we could have defined:
$$S := \C\{q^\lambda \mid \lambda \in \sigma\,\cap\, \ft^*_\Z\}, \qquad
S' := \C\{q^\lambda \mid \lambda \in \sigma\,\cap\, \Z\Sigma\},\and
S'' := \C\{q^\lambda \mid \lambda \in \N\Sigma_+\}.$$
We {\em a priori} have $S''\subset S' \subset S$.  In the case of the Springer resolution, the three rings coincide.
In the case of a hypertoric variety, $S = S'$ (Lemma \ref{basis}), but we do not know whether or not $S' = S''$.
We do not know of any examples (hypertoric or otherwise) for which the three rings do not coincide.

The constructions in this paper can be carried out using any of the three rings, but we prefer to use $S$ so that $\Spec S$ is a normal $T$-toric variety.
If $S \neq S'$, then $\Spec S'$ is a toric variety for a finite quotient of $T$.  If $S'\neq S''$, then $\Spec S''$ is not normal.
\end{remark}
}

Let $$R := S\otimes \Sym\Azo,$$
which we endow with a $\C[\hbar]$-algebra structure by setting $$x\, q^\lambda = q^\lambda \big(x + \hbar \langle\la,\bar x\rangle\big)$$
for all $\la\in\N\Sigma_+$ and $x\in \Azo$.  The ring $R$ is $\mathbb{N}$-graded, with $S$ in degree zero.

\begin{remark}\label{diff-ops}
For any $c\in H^2(\tXreg; \C)$, let $R^c$ be the quotient of $R$ by the ideal generated by $\theta - \hbar \langle \theta,c\rangle$ for all $\theta\in H_2(\tXreg; \C)\subset\Azo$, and let $R^c_{T}$ be the ring obtained from $R^c$ by
localizing from $\Spec S$ to $T$.  Then $R^c_{T}$ is (non-canonically) isomorphic to
the Rees algebra of differential operators on $T$, filtered by order.
If we choose a splitting of the quantization exact sequence \eqref{eq:quant-exact},
then we obtain a ring isomorphism by sending an element of $\ft$ to $\hbar$ times the corresponding translation invariant vector field on $T$.  For this reason, we think of $ R $ as a ring of differential operators with values in $H_2(\tXreg; \C)$.
\end{remark}

In the sections that follow,
we will be particularly interested in the localization $$\Sreg := S\!\left[\textstyle\frac{1}{1-q^{\la}}\;\Big{|}\; \la\in\Sigma_+\right].$$
We will also need to invert the same collection of elements in $R$; this requires care since $R$ is non-commutative. Let $\frak{S} \subset R$ be the multiplicative subset generated by $(1-q^{\la}), \la \in \Sigma_+$.

\begin{lemma} \label{orecondition}
The set $\frak{S}$ satisfies the Ore condition.  That is, for any $s \in \frak{S}$ and $r \in R$,
 there exists $s' \in \frak{S}$ and $r' \in R$ such that $s' r = r' s$.
\end{lemma}

\excise{
\begin{proof}
Lemma \ref{orecondition} holds essentially for the same reason that one can localize differential operators on a variety to any open subset. Namely, it is enough to check that the operator $\operatorname{ad} s$ is locally nilpotent for all $s \in \frak{S}$. We can filter $R$ by the degree in $x$ (placing both $S$ and $\C[\hbar]$ in degree zero). Then $\operatorname{ad} s$ strictly decreases degree with respect to this filtration, which shows it is locally nilpotent.
\end{proof}
\michtodo{I propose the above shortened proof (compare with the proof below). Do you think it's OK to cite Ginzburg's (somewhat casual) lecture notes for the nilpotency condition? (see http://www.math.ubc.ca/~cautis/dmodules/ginzburg.pdf, 1.3.8).}
}

\begin{proof}
For any homogeneous $r\in R$ let $\deg(r)\in \N$ be its degree. First, we claim that, for any $N\geq \deg(r)$, the commutator $[r, (1-q^{\la})^N]$ is right
divisible by $(1-q^\la)^{N-\deg(r)/2}$.  We prove this claim by induction on the degree of $r$.
When $\deg(r)=2$, we can immediately reduce to the case where $r=x\in\Azo$, and we have
\begin{equation} \label{degonecomm} [x, (1-q^{\la})^N] = -N\hbar\langle \la, \bar{x} \rangle q^\la (1-q^\la)^{N-1}. \end{equation}
Suppose our claim holds for all homogeneous elements of $R$ of degree less than $n=\deg(r)$.
We may reduce to the case where $r = r'x$ for some $x \in \Azo$ and $r' \in R$ with $\deg(r') = n-2$. Then
\begin{align*} [r,(1-q^{\la})^N] & = r'x(1-q^{\la})^N - (1-q^{\la})^N r'x \\
& = r'x(1-q^{\la})^N - r'(1-q^{\la})^N x + [r',(1-q^{\la})^N]x.
\end{align*}
Our inductive hypothesis tells us that there is some $r''\in R$ such that
$$[r', (1-q^\la)^N] = r''(1-q^\la)^{N-n/2+1},$$ thus
$$[r,(1-q^{\la})^N] = r'x(1-q^{\la})^N - r'(1-q^{\la})^N x + r''(1-q^\la)^{N-n/2+1} x.$$
The claim then follows from two applications of Equation \eqref{degonecomm}.

We return to our lemma.  We may assume that $r$ is homogeneous.  
Write $$s := \prod_{\la \in \Sigma_+} (1- q^{\la})^{N_\la},$$ and let $$s' := \prod_{\la \in \Sigma_+} (1- q^{\la})^{M},$$
where $M := \deg(r)/2  + \max\{N_{\la}\mid \la\in\Sigma_+\}.$
We have $s' r = rs' - [r, s']$, and our claim implies that $[r, s']$ is right divisible by
$\prod_{\la \in \Sigma_+} (1- q^{\la})^{N_\la} = s.$ This concludes the proof.
\end{proof}
We can thus define the Ore localization $R_\reg := \frak{S}^{-1} R$, which is isomorphic as a graded vector space to $\Sreg\otimes \Sym\Azo$.

\subsection{The module \boldmath{$M$}}
We endow $S \otimes \scrA_0$ with the structure
of an $\N$-graded left $R$-module by putting
$$x \cdot (q^\la \otimes a) :=  q^\la \otimes \big(x + \hbar\langle\la,\bar x\rangle\big) a \and
q^\mu \cdot (q^\la \otimes a) := q^{\la+\mu}\otimes a$$
for all $x\in\Azo$, $a\in\scrA_0$, and $\la,\mu\in\N\Sigma_+$.
Let
\begin{equation} \label{defJ} J := \sum_{\la\in\N\Sigma_+}S\cdot\left\{ 1\otimes a b - q^\lambda \otimes b a \mid a \in \scrA_\lambda, b \in \scrA_{-\lambda}\right\}
\subset S\otimes \scrA_0.\end{equation}
A priori, $J$ is a graded $S$-submodule of $S\otimes\scrA_0$.  Proposition \ref{Rsub} says that it is in fact an $R$-submodule.

\begin{proposition}\label{Rsub}
$J$ is an $R$-submodule of $S \otimes \scrA_0$.
\end{proposition}

\begin{proof}
It is sufficient to check that, if $x\in \Azo$, $\la\in\N\Sigma_+$, $a\in\scrA_\lambda$, and $b \in \scrA_{-\lambda}$, then
we have $x\cdot(1\otimes a b - q^\lambda \otimes b a)\in J$.  Indeed, we have
\begin{eqnarray*}x\cdot(1\otimes a b - q^\lambda \otimes b a)
&=& 1\otimes xab - q^\la\otimes \big(x + \hbar\langle\la,\bar x\rangle\big) ba\\
&=& 1\otimes xab - q^\la\otimes xba + q^\la\otimes [x,b] a \\
&=& 1\otimes xab - q^\la \otimes bxa,
\end{eqnarray*}
which is an element of $J$.
\end{proof}

Our main object of study will be the graded $R$-module $$M := \left(S \otimes \scrA_0\right) / J,$$ which we call the {\bf D-module of graded
traces} (see Section \ref{sec:traces} for the motivation behind this terminology).  We will be particularly
interested in the localization $$\Mreg := \Rreg\otimes_R M.$$

\begin{example}\label{sl2example1}
	We illustrate these constructions when $X$ is the Kleinian singularity of type $A_1$, or (equivalently) the nilpotent cone in $\sl_2^*$.
	Choose a basis $\{\hbar,a_1,a_2\}$ for $\Azo$ such that
	$\Sigma_+\subset\ft^*$ consists of a single element that evaluates to 1 on both $\bar a_1$ and $\bar a_2$.
	Then $S = \C[q]$ and $R$ is generated over $\C[\hbar]$ by $q$, $a_1$, and $a_2$.  The classes $a_1$ and $a_2$ commute with each other,
	and  $ a_i q = q (a_i + \hbar)$.
	The $\C[\hbar]$-algebra $ \scrA $ has generators $ r_+, r_-, a_1, a_2 $ and relations
	\begin{gather*}
	r_+ r_- = a_1 a_2, \quad r_-r_+ = (a_1 + \hbar)(a_2 + \hbar) \\
	[a_i, r_+] = \hbar r_+, \quad [a_i, r_-] = -\hbar r_-, \quad [a_1,a_2]=0.
	\end{gather*}
	This algebra is an example of a hypertoric enveloping algebra (Section \ref{sec:hea}),
	and it is also the Rees algebra of the enhanced enveloping algebra of $\sl_2$.
	To see this, we identify $r_+$ with $E$, $-r_-$ with $F$, $a_1 + a_2 + \hbar$ with $H$, and $a_1-a_2$ with
	the square root of the central element
	$ C = 2EF + 2FE + H^2 + \hbar^2 $.
	
	Now let us compute the module $ M $.  We have $ \scrA_0 = \C[\hbar, a_1, a_2] $, which implies that $ S \otimes \scrA_0 $ is a free $R$-module
	of rank 1.  So it remains to compute the left ideal $ J \subset S\otimes\scrA_0\cong R$.
	By the definition, $J$ contains the element
	$$ r_+ r_- - qr_- r_+ = a_1 a_2 - q (a_1 + \hbar)(a_2 + \hbar) = a_1a_2(1-q).$$
	It is a special case of Proposition \ref{JJ'} that $J$ is in fact generated by this single element.
\end{example}

\subsection{Specializations and Hochschild homology}\label{sec:specializations}
We will now relate some specializations of $ M $ to the degree zero Hochschild homology of $\scrA$.  Recall that if $ A $ is any ring and $ B $ is an $ (A, A)$-bimodule, then $$
\HH_0(A, B) := B \big{/} \C \{ a b - b a : a \in A, b \in B \}.
$$  
We will write $\HH_0(A) := \HH_0(A,A) $ where $ A $ is the regular bimodule.

We begin by considering the specialization $$M_0 := \C_0\otimes_S M, $$ where the map $S\to\C_0$ is given by setting $q$ equal to zero
(or, equivalently, by evaluation at the unique $T$-fixed point of the toric variety $\Spec S$).  Note that $ M_0 $ carries the structure of a $ \Sym \Azo$ module since $ \Azo $ lies in the centre of $ \scrA_0 $.  This specialization is closely related to the algebra 
 $$B(\scrA) := \scrA_0\Big{/} \sum_{\langle\mu,\xi\rangle > 0}
\left\{a b \mid a \in \scrA_\mu, b \in \scrA_{-\mu}\right\},$$
which was introduced in \cite[Section 5.1]{BLPWgco} to study the representation theory $\scrA$, and which also
appears in the equivariant Hikita conjecture \cite[Conjecture 8.9]{KTWWY}.

\begin{remark}\label{signs}
	The definition of $B(\scrA)$ in \cite[Section 5.1]{BLPWgco} uses the inequality $\langle\mu,\xi\rangle < 0$
	rather than $\langle\mu,\xi\rangle > 0$.  This means that the algebra that we study in this paper is related
	to category $\cO$ for the parameter $-\xi$ rather than $\xi$.
\end{remark}

\begin{proposition}\label{M0}
	We have an isomorphism $M_0 \cong \HH_0(B(\scrA))$ of $\Sym\Azo$-modules.
\end{proposition}

\begin{proof}
	Let $$\tilde J := \sum_{0\neq \la\in\N\Sigma_+}S\cdot\left\{ 1\otimes a b - q^\lambda \otimes b a \mid a \in \scrA_\lambda, b \in \scrA_{-\lambda}\right\}
	\subset S\otimes \scrA_0.$$
	Thus $\tilde J$ is defined in the same way as $J$, except that we do not allow $\la=0$.  Let $\tilde M := R/\tilde J$ and $\tilde M_0 := \C_0\otimes_S \tilde M$.
	
	We claim that $\tilde M_0$ is isomorphic to $B(\scrA)$.  Indeed, $\tilde M_0$ is the quotient of $\scrA_0$
	by the ideal $\tilde J_0 := \C_0 \otimes_S \tilde J$, and it is clear that $\tilde J_0$
	is contained in the kernel of the surjection from $\scrA_0$ to $B(\scrA)$.
	We need to show that, if $\langle\mu,\xi\rangle > 0$, $a \in \scrA_\mu$, and $b \in \scrA_{-\mu}$, then $ab\in \tilde J_0$.
	To see this, write $$a = \sum y_i z_i $$ as in Lemma \ref{putrootinfront}.
	Then $ab = \sum (y_i z_i) b = \sum y_i (z_i b) \in \tilde J_0$.
	
	It remains to show that $M_0$ is isomorphic to the degree zero Hochschild homology of $\tilde M_0$.  Now $ M_0 = \scrA_0 / J_0 $, where $ J_0 := \C_0 \otimes_S J$.  We observe that the difference between $J_0$ and $\tilde J_0$ is that $J_0$ contains
	the linear span of $ab - ba$ for all $a,b\in\scrA_0$, which descends to the linear span of arbitrary commutators in the ring $\tilde M_0$.
\end{proof}

\begin{remark}\label{Nak-abelian}
	The equivariant Hikita conjecture \cite[Conjecture 8.9]{KTWWY} states that,
	in the presence of symplectic duality,
	$B(\scrA)$ is isomorphic to the equivariant cohomology ring of the dual variety, which is concentrated in even degree.
	If this conjecture holds, then $B(\scrA)$ is commutative, and therefore equal to its own degree zero Hochschild homology.
	Thus, if we assume that $X$ has a symplectic dual for which the equivariant Hikita conjecture holds, 
	then Proposition \ref{M0} simply says that $M_0$ is isomorphic to $B(\scrA)$.
\end{remark} 

Now we go in the opposite direction and consider the $ \C[\hbar] $-algebras 
$$M_T := \cO(T) \otimes_S M \and M_1 := \C_1\otimes_S M $$
obtained by specializing $M$ over the full torus $ T $ and at the identity element of $T$.

\begin{remark}
We note that $M_0$ can be recovered from $\Mreg$, since $1-q^{\la}$ does not evaluate to zero at the fixed point of $\Spec S$.  On the other hand, $M_1$
cannot be recovered from $\Mreg$, since $1-q^{\la}$ {\em does} evaluate to zero at the identity element of $T$.
\end{remark}

Consider the algebra $ \scrA^{\hbar = 1} := \scrA \otimes_{\C[\hbar]} \C[\hbar]/\langle\hbar-1\rangle $.  Define a bimodule over this algebra with underlying vector space $ \cO(T) \otimes \scrA^{\hbar = 1} $, where the left action is given by $ a \cdot f \otimes b = f q^{-\lambda} \otimes ab $ for $ a \in \scrA^{\hbar = 1}_\lambda $, $ f \in \cO(T) $, and $ b \in \scrA^{\hbar = 1} $, and where the right action is just given by right multiplication.

\begin{remark}
	The definition of this bimodule is motivated by a recent paper by Etingof-Stryker \cite{ES}.  They study twisted traces, which are closely related to the graded traces we study in this paper.
\end{remark}

Consider the $\scrA^{\hbar = 1}$-modules
$$M_T^{\hbar = 1} := \scrA^{\hbar = 1}\otimes_\scrA M_T \and M_1^{\hbar = 1} := \scrA^{\hbar = 1}\otimes_\scrA M_1.$$

\begin{proposition}\label{M1}
We have vector space isomorphisms
$$M_T^{\hbar = 1} \cong \HH_0(\scrA^{\hbar = 1}, \cO(T) \otimes \scrA^{\hbar = 1}) \and M_1^{\hbar = 1} \cong \HH_0(\scrA^{\hbar = 1}).$$
\end{proposition}

\begin{proof}
We first observe that the second isomorphism follows from the first, so we only need to prove the first.
Using the invertibility of $q^\lambda $ in $\cO(T) $, we see that
	$$
	\HH_0(\scrA^{\hbar = 1}, \cO(T) \otimes \scrA^{\hbar = 1}) = \cO(T) \otimes  \scrA^{\hbar = 1} \Big{/} \sum_{\lambda \in \ft_\Z^*} \cO(T) \Big\{ 1 \otimes  ab - q^\lambda \otimes ba \;\Big{|}\; a \in \scrA^{\hbar = 1}_\lambda, b \in \scrA^{\hbar = 1} \Big\}.
	$$
	
Given $\la\neq 0$, choose $ x \in \Azo$ such that $ \langle \lambda, \bar x \rangle \ne 0 $.
If $ b \in \scrA^{\hbar = 1}_\lambda $, then the commutator of $b$ with the image of $x$ in $\scrA^{\hbar = 1}_0$ will be a nonzero
multiple of $b$, which implies that the image of $ 1 \otimes b $ in $   \HH_0(\scrA^{\hbar = 1}, \cO(T) \otimes \scrA^{\hbar = 1})$
will be trivial.  It follows that
$\HH_0(\scrA^{\hbar = 1}, \cO(T) \otimes \scrA^{\hbar = 1})$ is a quotient of $ \cO(T) \otimes \scrA^{\hbar = 1}_0 $.

Let 
$$
J^{\hbar = 1}_T :=  \sum_{\la\in\N\Sigma_+} \cO(T) \left\{ 1\otimes a b - q^\lambda \otimes b a \mid a \in \scrA^{\hbar = 1}_\lambda, b \in \scrA^{\hbar = 1}_{-\lambda}\right\}
\subset \cO(T) \otimes \scrA^{\hbar = 1}_0,
$$
so that $M^{\hbar = 1}_T = \cO(T) \otimes \scrA^{\hbar = 1}_0 / J^{\hbar = 1}_T.$  It is clear that $ J^{\hbar = 1}_T $ is contained in the kernel of the map $$\cO(T) \otimes \scrA^{\hbar = 1}_0 \rightarrow \HH_0(\scrA^{\hbar = 1}, \cO(T) \otimes \scrA^{\hbar = 1}),$$ so it suffices to prove the reverse containment.

The $\N$-grading on $ \scrA $ descends to a filtration of $ \scrA^{\hbar = 1} $, where $(\scrA^{\hbar = 1})^k$
is the image of $\scrA^{\leq k}$ in $ \scrA^{\hbar = 1} $.
We will show that, if $a \in (\scrA^{\hbar = 1}_\mu)^k$ and $b \in (\scrA^{\hbar = 1}_{-\mu})^l$, then $1 \otimes ab - q^\mu \otimes ba \in J^{\hbar = 1}_T$.
We will assume without loss of generality that $k\leq l$ and proceed by induction on $k$.

Assume first that $\langle\mu,\xi\rangle > 0$.
As in Lemma \ref{putrootinfront}, write $a = \sum y_i z_i$ for some $y_i \in \scrA^{\hbar = 1}_{\lambda_i}$ and $z_i \in (\scrA^{\hbar = 1}_{\mu - \lambda_i})^{q_i}$ with $\lambda_i \in \Sigma_+$ and $q_i < k$ for all $i$. 
Then for each $ i$, we have
$$
1 \otimes y_i z_i b - q^\mu \otimes b y_i z_i = 1\otimes y_i z_i b - q^{\lambda_i} \otimes z_i b y_i + q^{\lambda_i} (1\otimes z_i b y_i - q^{\mu - \lambda_i} \otimes b y_i z_i).
$$
By the definition of $ J^{\hbar = 1}_T $, we see that $1\otimes y_i z_i b - q^{\lambda_i} \otimes z_i b y_i \in J^{\hbar = 1}_T $.
Our inductive hypothesis tells us that $ 1 \otimes z_i b y_i - q^{\mu - \lambda_i} \otimes b y_i z_i \in J^{\hbar = 1}_T $, as well.
Thus $1 \otimes y_i z_i b - q^\mu \otimes b y_i z_i \in J^{\hbar = 1}_T $ for all $i$, and therefore $1 \otimes ab - q^\mu \otimes ba \in J^{\hbar = 1}_T $ as desired.

If $\langle\mu,\xi \rangle< 0$, then we can replace $\xi$ with $-\xi$, which does not affect the definitions of either $J_T^{\hbar = 1}$
or $\HH_0(A, \cO(T)\otimes A)$, and thus reduce to the previous case.
Finally, suppose that $\langle\mu,\xi\rangle = 0$.
If $\mu\neq 0$, then we can perturb $\xi$ so that this
is no longer the case without changing $\Sigma_+$.  This again does not affect the definitions of either $J_T^{\hbar = 1}$ or $\HH_0(A, \cO(T)\otimes A)$, and we may again reduce to the previous cases.
If $\mu=0$, then the fact that $1 \otimes ab - q^\mu \otimes ba = [a,b] \in J_T^{\hbar = 1}$ is immediate from the definition of $J$.
\end{proof}

\subsection{Finite generation}
Recall that $ \Mreg $ is a module over $ \Rreg$,
and $\Rreg$ may be regarded as a subring of differential operators on $\Spec\Sreg$ with values in $H_2(\tXreg; \C)$ (Remark \ref{diff-ops}).
In particular $ \Rreg $ is generated by three types of elements: the ``vector fields'' $\ft$, the ``functions" $\Sreg$, and the ``values" 
$H_2(\tXreg; \C)$.  The following result says that $ \Mreg $ is finitely generated over just the functions and the values.

\begin{proposition} \label{finitely generated}
If $ M_0 $ is finitely generated as a module over $\Sym H_2(\tXreg; \C)\otimes \C[\hbar]$, then $\Mreg$ is finitely generated as a module over $\Sreg \otimes \Sym H_2(\tXreg; \C) \otimes \C[\hbar]$.
\end{proposition}
\begin{proof}
Choose elements $x_1,\ldots,x_d\in\scrA_0$ whose images generate $M_0$ as a module over $ \Sym H_2(\tXreg; \C)\otimes\C[\hbar] $,
and let $\Mreg'\subset\Mreg$ be the submodule spanned by the images of $x_1,\ldots,x_d$ in $\Mreg$.
We will show that $\Mreg'=\Mreg$.  To do this, we
will prove by induction that, for any natural number $m$, the degree $m$ parts of $\Mreg$ and $\Mreg'$ coincide. 
The base case $m=0$ holds because $\Mreg^0=\Sreg$.

Let $ a \in \scrA_0^m$.  Since $x_1,\ldots,x_d$ generate $M_0$, we may choose elements $r_1,\ldots,r_d\in \Sym H_2(\tXreg; \C)\otimes \C[\hbar]$,
$\la_1,\ldots,\la_e\in\N\Sigma_+$, and $a_1,b_1,\ldots,a_e,b_e\in\scrA$ with $a_j\in\scrA_{\la_j}$ and $b_j\in\scrA_{-\la_j}$,
such that
$$
a = \sum_{i=1}^d r_ix_i+\sum_{j=1}^e a_jb_j \in\scrA_0^m.
$$
By Lemma \ref{putrootinfront}, we may also assume that $\la_j\in\Sigma_+$ for all $j$.
It now suffices to show that, for each $j$, the image of $a_jb_j$ in $\Mreg$ lies in $\Mreg'$.

In $ \Mreg $, we have $$a_jb_j = q^{\la_j}b_ja_j.$$
On the other hand, since the quotient of $\scrA$ by $\scrA\hbar$ is commutative, there exists an element $c\in\scrA^{m-2}_0$
such that $[b_j,a_j] = \hbar c$.  Combining this with the previous equation, we see that in $ \Mreg $,
$$
a_j b_j = \frac{q^{\la_j}}{1 - q^{\la_j}} \hbar c.
$$
 By our inductive hypothesis, the image of $c$ lies in $\Mreg'$,
thus so does the image of $a_jb_j$.
\end{proof}

By Proposition \ref{M0} and \cite[Proposition 5.1]{BLPWgco}, $M_0$ is finitely generated as a module over $\Sym H_2(\tXreg; \C)\otimes \C[\hbar]$
whenever $\tX$ is smooth and the action of $T$ on $\tX$ has isolated fixed points.
We thus have the following corollary.
\begin{corollary} \label{isolfixedfinitegeneration}
If $\tX$ is smooth and $T$ acts on $\tX$ with isolated fixed points, then $\Mreg$ is finitely generated as a module over $\Sreg \otimes \Sym H_2(\tX; \C) \otimes \C[\hbar]$.
\end{corollary}

\subsection{Traces} \label{sec:traces}
Given an $\C$-algebra $ A$, its degree zero Hochschild homology is used to study traces of finite-dimensional representations.  Indeed, if $ V $ is an $A$-module which is finite-dimensional as a $\C$-vector space, then the trace map $ A \rightarrow \C $ given by $ a \mapsto \tr(a |_V) $ factors through $\HH_0(A)$.
Our algebra $ \scrA $ carries a grading and the $R$-module $ M $ can be thought of as a graded version of the degree zero Hoschshild homology of $\scrA$ (if we set $ \hbar = 1 $ and restrict to $T$, this is Proposition \ref{M1}).  Now, we will define a notion of graded traces for $ \scrA$-modules and prove that these graded traces factor through $ M $ (Proposition \ref{lemma characters}).

Let $ \Pi $ denote the set of linear maps $ \mu : \Azo \rightarrow \C \hbar $ such that $ \mu(\hbar) = \hbar $.  There is a free action of $ \ft^*_\Z $ on $ \Pi $ given by $ (\lambda + \mu)(x) = \mu(x) + \hbar \lambda(\bar x) $, where $ \lambda \in \ft^*_\Z$ and $ \mu \in \Pi$ .
Let $$N := \left\{\sum_{\mu \in \Pi} p_\mu q^\mu \;\Big{|}\; p_\mu \in \C[\hbar]\right\}.$$
Note that $N$ is similar to $S\otimes\C[\hbar]$, but it is much bigger;
we allow the exponent of $q$ to live in $\Pi$ rather than in $\N\Sigma_+$,
and we also allow infinite rather than finite sums.
We endow $N$ with the structure of a left $R$-module by putting
$$q^\lambda \cdot q^\mu = q^{\lambda + \mu}\and x \cdot q^\mu
= \mu(x) q^\mu$$
for all $\la\in\N\Sigma_+$, $\mu\in \Pi$, and $x\in \Azo$.

Let $ V $ be a graded $ \scrA $-module.  For any $ \mu \in \Pi $, let
$$ V_\mu := \{ v \in V \mid x\cdot v = \mu(x) v\; \text{ for all } x \in \Azo \}.
$$
Then for all $ a \in \scrA_\lambda $ and $ v\in V_\mu $, we have $ a \cdot v \in V_{\lambda+\mu}$.
If each $ V_\mu $ is a free $\C[\hbar]$-module of finite rank, we call $V$ {\bf reasonable}.
We define the {\bf graded trace} of a reasonable $\scrA$-module $ V $ to be the map
\begin{equation*}
\begin{aligned}
\tr_V : \scrA_0 &\rightarrow  N\\
a &\mapsto \sum_{\mu \in \Pi} \tr(a|_{V_\mu}) q^\mu.
\end{aligned}
\end{equation*}
In particular $ \chi_V := \tr_V(1) $ is the generating function for the ranks of the weight spaces as $\C[\hbar]$-modules, 
which we will refer to as the {\bf character} of the representation.

\begin{proposition} \label{lemma characters}
The graded trace descends to an $\N$-graded $R$-module map $\tr_V:M \rightarrow N$.
\end{proposition}

\begin{proof}
We need to show that $J$ is contained in the kernel of $\tr_V$.  More concretely, if
$ a \in \scrA_\lambda$ and $b \in \scrA_{-\lambda} $, we need to show that $ \tr_V(ab) = q^\lambda \tr_V(ba) $.
Pick an element $ \mu \in \Pi $ and consider the maps $a : V_\mu \rightarrow V_{\la + \mu} $ and $ b : V_{\lambda+\mu} \rightarrow V_\mu $.  Since these are linear maps between finite-rank free $\C[\hbar]$-modules, we have
$$
\tr(ba|_{V_{\mu}}) = \tr(ab|_{V_{\lambda+\mu}}),
$$
and therefore
$$q^\la\tr_V(ba) = q^\la\sum_\mu\tr(ba|_{V_\mu})q^\mu = q^\la\sum_\mu\tr(ab|_{V_{\lambda+\mu}})q^\mu
= \sum_\mu\tr(ab|_{V_{\lambda+\mu}})q^{\la+\mu} = \tr_V(ab).$$
This completes the proof.\end{proof}

\begin{remark}\label{solutions}
Fix a splitting of the quantization exact sequence \eqref{eq:quant-exact}. Recall from Remark \ref{diff-ops} that, given $c \in H^2(\tXreg; \C)$, the ring $R^c_{T}$ obtained by 
localizing to $T$ and killing
the ideal generated by $\theta - \hbar \langle \theta, c \rangle$ for all $\theta \in H_2(\tXreg; \C)$ is isomorphic
to the Rees algebra of the ring of differential operators on $T$. 
Let $M^{c} := R^{c} \otimes_R M$, 
$$\Pi^{c} := \left\{\mu\in\Pi\;\Big{|}\; \text{$\mu(\theta) = \hbar \langle \theta,c\rangle$ for all $\theta \in H_2(\tXreg; \C)$}\right\},$$
The splitting of \eqref{eq:quant-exact} identifies $\Pi^{c}$ with $\ft^*$. Define
$$N^c := \left\{\sum_{\mu \in \Pi^c} p_\mu q^\mu \;\Big{|}\; p_\mu \in \C[\hbar]\right\}.$$
We endow $N^c$ with an $R^{c}$-module structure as above. Then for any reasonable $V$, the graded trace map descends to a $R^{c}$-module map
$$\tr^{c}_V: M^{c} \to N^c.$$
We may regard this map as a ``solution" to the D-module $M^{c}$ on $T$. Note that it depends on the choice of a splitting of \eqref{eq:quant-exact}, i.e. of a quantum co-moment map.
\end{remark}

\begin{remark} \label{posreasonable}
Let $\Nreg \subset N$ be the set of all $\sum p_\mu q^\mu$ such that for all $\la\in\Sigma_+$ and all $\mu\in\Pi$, we have $ p_{\mu - n\la} = 0 $ for all $n>\!\!> 0$.
This is an $R$-submodule on which $1-q^\la$ acts invertibly
for all $\la\in\Sigma_+$, and therefore the action of $R$ on $\Nreg$ extends to an action of $\Rreg$.
We say that a graded $ \scrA$-module $V $ is {\bf positively reasonable} if it is reasonable and $\tr_V(1)\in\Nreg$.
The the graded trace map extends to an $ \Rreg $-module map $\tr_V: \Mreg \rightarrow  \Nreg$.
\end{remark}

\begin{example} \label{sl2example2}
We continue with Example \ref{sl2example1}.  
We have $\Azo = \C\{a_1,a_2,\hbar\}$.  
For any $\mu\in\Pi$, let $\mu_i := \mu(a_i)/\hbar$, and write $q^\mu = q_1^{\mu_1}q_2^{\mu_2}$.
Let $ V $ be a positively reasonable representation of $ \scrA $
on which the central element $ a_2 - a_1 \in\scrA$ acts by a scalar $ z\hbar $
for some complex number $z$; this implies that
$ V_{\mu} = 0 $ unless $ \mu_2 - \mu_1 = z $.
By Proposition \ref{lemma characters}, we have $ a_1 a_2 (1-q) \chi_V = 0 \in N$,
which implies that there exist $ p_1, p_2 \in \Z$ such that $$ (1-q) \chi_V = p_1 q_1^{-z} + p_2 q_2^{z}.$$
Then
$$ \chi_V = \frac{p_1 q_1^{-z} + p_2 q_2^{z}}{1-q} = p_1 \sum_{n=0}^\infty q_1^{n-z}q_2^n + p_2 \sum_{n=0}^\infty q_1^n q_2^{n+z},
$$
which is reminiscent of the Weyl character formula.
If $V_1$ and $V_2 $ are the Verma modules for $ \scrA$ with central character $z\hbar$, we have
$$
\chi_{V_1} =  \sum_{n=0}^\infty q_1^{n-z}q_2^n \and \chi_{V_2} =  \sum_{n=0}^\infty q_1^n q_2^{n+z}.
$$
If $ z$ is a positive integer, there is a finite-dimension module $ V $  with $p_1 = 1$ and $p_2 = -1$, so that
$$\chi_V = \sum_{n=0}^{z-1} q_1^{n-z} q_2^{n}.$$ Similarly, if $z$ is a negative integer, there is a finite dimensional module $V$ with $p_1 = -1$ and $p_2=1$,
so that $$\chi_V = \sum_{n=0}^{1-z} q_1^{n} q_2^{n+z}.$$
\end{example}

\subsection{The rank of \boldmath{$M_\reg$}}\label{sec:rank}
Assume for the remainder of this section that $X$ admits a $(T \times \mathbb{G}_m)$-equivariant projective symplectic resolution $\tX$ with isolated $T$-fixed points.
For each fixed point $x \in \tX^T$, we may define a local version $\scrA_x$ of $\scrA$
by quantizing the tangent space of $x$.  The inclusion of a formal neighborhood of $x$ into $\tX$ induces 
a $T$-equivariant surjection $\scrA^{ }\to\scrA_x^{ }$ \cite[Section 5.1]{BLPWgco}.
Given $c \in H^2(\tX; \C)$, let $\scrA_x^{c}$ be the quotient of $\scrA_x$ by the ideal generated by the central elements
$\theta - \hbar \langle \theta, c \rangle$ for all $\theta \in H_2(\tX; \C)$.
We note that $\scrA_x^{c}$ is isomorphic to the Rees algebra of a Weyl algebra.  

We use the cocharacter $\xi\in\ft_\Z$ to induce a $\Z$-grading on $\scrA_x^{c}$, and denote by
$\scrA_{x,-}^{c}$ the direct sum of the non-positive weight spaces. 
Consider the left $\scrA_x^{c}$-module
$$\Theta^c_x := \scrA_x^{c} \otimes_{\scrA_{x, -}^{c}} B(\scrA_x^{c}).$$
Then $\Theta^c_x$ is positively reasonable in the sense of Remark \ref{posreasonable} as a graded $\scrA$-module, and we have
\begin{equation} \label{vermatrace} \tr^{c}_{\Theta^c_x}(1\otimes 1) = q^{w^c_{x}} \prod_{i} \frac{1}{1-q^{\alpha_i}}, \end{equation}
where the elements $\a_i\in\Sigma_+$ are the positive weights of the action of $T$ on the tangent space $T_x\tX$,
and $w^c_x\in\Pi^c$ is the linear map $ w^c_x : \Azo \rightarrow \C \hbar $ given by restricting $$ \scrA_0 \rightarrow (\scrA^c_x)_0 \rightarrow B(\scrA^c_x) \cong \C[\hbar] $$ to degree 2 \cite[Proposition 5.20]{BLPWgco}.

Let $M(\scrA_x) :=  S \otimes \scrA_{x, 0} /J_x$, where $J_x$ is defined exactly as in Equation \eqref{defJ}. 
Define $M(\scrA_x)^c := R^c\otimes_R M(\scrA_x)$ as in Section \ref{sec:traces},
and define $M(\scrA_x)_{\reg}$ and $M(\scrA_x)^c_{\reg}$ in the obvious manner.

\begin{lemma}\label{local M}
The natural map
\[ \pointrest_x: \Sreg\otimes\C[\hbar] \to M(\scrA_x)^{c}_{\reg} \]
is an isomorphism.
\end{lemma}

\begin{proof}
We begin by showing that $\pointrest_x$ is a surjection. The algebra $\scrA_x$ is generated over $\C[\hbar]$ by positive degree elements $z_i, w_i$ 
with weights $\alpha_i,-\alpha_i$ and relations $[w_i, z_j] = \pm\delta_{ij}\hbar$.
By induction, we suppose that all elements of $M(\scrA_x)_\reg$ of degree strictly less than $d$ lie in the image of $\pointrest_x$. 
Let $m\in M(\scrA_x)^{c}_{\reg}$ be an element of degree $d>0$. 
Then there exist elements $a_i\in \scrA_x$ of weight $-\a_i$ and degree strictly less than $d$, $m'\in M(\scrA_x)^{c}_{\reg}$ of degree $d-2$ and $s_i \in \Sreg$ such that $$m = \hbar m' + \sum s_i \otimes z_i a_i.$$ 
It will therefore suffice to show that each $1\otimes z_ia_i$ lies in the image of $\pointrest_x$. 
We have
$$1\otimes z_i a_i = q^{\alpha_i} \otimes a_i z_i = q^{\alpha_i} \otimes (z_i a_i  + [a_i, z_i]),$$
which implies that $$(1-q^{\a_i})\otimes z_i a_i = q^{\a_i} \otimes  [a_i, z_i]$$
and therefore $$1\otimes z_i a_i = \frac{q^{\a_i}}{1-q^{\a_i}} \otimes [a_i, z_i].$$
Since $[a_i, z_i]$ is a multiple of $\hbar$, our inductive hypothesis implies that this element lies in the image of $\pointrest_x$.

It remains to show that the kernel of $\pointrest_x$ is trivial. 
The composition $$\tr^{c}_{\Theta^c_x}\circ\;\pointrest_x:\Sreg\otimes\C[\hbar]\to N^c$$
takes $s$ to $s\cdot \tr^{c}_{\Theta^c_x}(1)$, and this map is clearly injective by Equation \eqref{vermatrace}. 
Thus $\pointrest_x$ must be injective, as well.
\end{proof}

Let $\Treg := T \cap \Spec\Sreg$, and let
$$M_{T_\reg} := \cO(T_\reg)\otimes_{S} M.$$  
Let $M^{c,1}$ be the specialization of $M^c$ at $\hbar = 1$,
so that $M^c$ is isomorphic to the Rees module of the filtered module $M^{c,1}$.  Define $M^{c,1}_{\Treg}$ similarly.
For the remainder of this section, we will fix a splitting of the quantization exact sequence, so that
$M^{c,1}_{T_\reg}$ defines a $D(T_\reg)$-module by Remark \ref{diff-ops}. 
This choice of splitting is not essential in any way, but it is convenient because it allows us to use the language of D-modules.

Since $M^{c,1}_{T_\reg}$ is finitely generated over $\mathcal{O}(\Treg)$
(Corollary \ref{isolfixedfinitegeneration}), it defines a nonsingular D-module over $T_\reg$, and therefore comes from a vector bundle with flat connection; we will use the same notation
$M^{c,1}_{T_\reg}$ to refer to this vector bundle. 
The main result of this section, Corollary \ref{mainresultrank}, computes the rank of this vector bundle under certain assumptions.

We will say that $c$ is {\bf generic} if it satisfies the conditions of \cite[Theorem 5.12]{BLPWgco} 
and \cite[Lemma 5.21]{BLPWgco} for each 
fixed point $x \in \tX^T$.

\begin{proposition} \label{prop:ranksmaller}
If $c$ is generic, then
$\displaystyle \rk M^{c,1}_{T_\reg} \leq |\tX^T|.$ \end{proposition}

\begin{proof}
We have a coherent sheaf $M^{c,1}$ on $\Spec S$ whose restriction to $T_\reg$ is equal to the vector bundle $M^{c,1}_{T_\reg}$,
thus the rank of this vector bundle is bounded above by the dimension of the fiber of $M^{c,1}$ at the point $0\in\Spec S$. 
By Proposition \ref{M0}, this fiber is isomorphic to $\HH_0(B(\scrA))^{c,1}$, which is a quotient of $B(\scrA)^{c,1}$.  
For generic $c$, the dimension of $B(\scrA)^{c,1}$ is equal to $|\tX^T|$ by \cite[Proposition 5.3]{BLPWgco}.
\end{proof}

Our splitting of the quantization exact sequence identifies $\Pi^c$ with $\ft^*$, thus $q^{w^c_x}$ defines a multi-valued function on $T_\reg$ with monodromy $e^{2\pi i w^c_x(\tau)}$ around any loop $\tau \in \pi_1(T) \cong \ft_\Z$. 
Consider the rank one D-module $\mathcal{O}(T_\reg)q^{w^c_x}$ generated by $q^{w^c_x}$. In other words, it is the quotient of $D(T_\reg)$ by the left ideal $\langle \partial_u - w^c_x(u) \rangle$ for $u \in \frak{t}$.
Equation \eqref{vermatrace} and Lemma \ref{local M} together imply that the map $\tr^{c}_{\Theta^c_x}:M^c\to N^c$
descends to a nonzero map $$\tr^{c,1}_{\Theta^c_x}:M^{c,1}_{\Treg}\to \mathcal{O}(T_\reg)q^{w^c_x}.$$
Since the target is a simple $D$-module, it must be surjective.

Consider the sum
\begin{equation*} \label{dualmap}  
\theta^c: M_{T_\reg}^{c,1}
\to \bigoplus_{x\in\tX^T} \mathcal{O}(T_\reg)q^{w^c_x}
\end{equation*}
of these maps over all fixed points.
If the cosets $w^c_x+\ft^*_\Z$ are distinct, then the summands of the target
are non-isomorphic simple $D$-modules and the map $\theta^c$ must therefore be surjective. 
We will show that, under additional hypotheses, this is the case.
Specifically, in Appendix \ref{appendixWebster} we define maps $\rho_{x,y}^* : H^2(\tX ; \C) \rightarrow \ft^*$ for any pair of fixed points $x, y \in \tX^T$, and we ask that these maps be nonzero for $x \neq y \in \tX^T$. This holds, for instance, for hypertoric varieties attached to unimodular arrangements and for the Springer resolution, but not for the balanced Hilbert scheme of $n$ points in the plane. 
\begin{lemma} \label{distinctweights}
Suppose that the map $\rho_{x,y}$ in nonzero for all $x\neq y\in\tX^T$.
For $c$ in a non-empty analytic open subset, the cosets $w^c_x+\ft^*_\Z$ are distinct.
\end{lemma}

\begin{proof}
For any two fixed points $x\neq y\in\tX^T$,
Theorem \ref{thm:differenceofhighestweights} (with $ \hbar = 1 $) tells us that
$$w^c_x - w^c_y =  \rho_{x,y}^*(c) + \chi_{x}-\chi_{y}, $$
where $ \chi_x \in \ft^* $ is independent of $ c $.
Since we require $\rho_{x,y}$ to be nonzero, the set of $c$ for which any one of these differences lies in $\ft^*_\Z$ forms a discrete union of affine subspaces of codimension $\geq 1$.
\end{proof}
\begin{proposition} \label{prop:rankbigger}
Suppose that the map $\rho_{x,y}$ in nonzero for all $x\neq y\in\tX^T$. For all $c \in H^2(\tX; \C)$, we have
$$\rk M^{c,1}_{T_\reg} \geq |\tX^T|.$$ \end{proposition}

\begin{proof}
Since $M_{T_\reg}$ is coherent over $\H^2(\tX; \C) \times T_{\reg} \times \Spec \C[\hbar]$, 
it is enough to prove this for $c$ lying in a nonempty analytic open subset. By Lemma \ref{distinctweights}, there exists such a subset 
for which the cosets $w^c_x+\ft^*_\Z$ are distinct.  For $c$ in this subset, the map $\theta^c$ must be surjective, thus 
$\rk M_{T_\reg}^{c,1} \geq |\tX^T|$. 
\end{proof}

Since $M_{\reg}^{c,1}$ is coherent over $\Spec S_\reg$ and $T_\reg$ is dense in $\Spec S_\reg$, we have the following mild strengthening of Proposition \ref{prop:rankbigger}.

\begin{corollary} 
Suppose that the maps $\rho_{x,y}$ are nonzero for all $x\neq y\in\tX^T$.
For all $c$ and any $q \in \Spec S_\reg$, the fiber of $M^{c,1}_{\reg}$ at $q$ has dimension greater than or equal to $|\tX^T|.$
\end{corollary}

Combining Propositions \ref{prop:ranksmaller} and \ref{prop:rankbigger}, we obtain the main result of this section.

\begin{corollary} \label{mainresultrank}
Suppose that the map $\rho_{x,y}$ in nonzero for all $x\neq y\in\tX^T$. For $c\in H^2(\tX;\C)$ generic, $\rk  M^{c,1}_{T_\reg} = |\tX^T|$, the map $\theta^c$ is an isomorphism,
and thus we have an isomorphism of D-modules 
$$M_{T_\reg}^{c,1} \cong \bigoplus_{x \in \tX^T} \mathcal{O}(T_\reg)q^{w^c_x}.$$
\end{corollary}

\begin{remark}
Many interesting symplectic resolutions $\tX$ with isolated fixed-points, such as $\tX = \operatorname{Hilb}_n(\C^2)$, do not have distinct restriction maps $H^2_T(\tX ; \C) \to H^2_T(x ; \C)$. In this case the maps $ \rho_{x,y} $ defined in the appendix vanish.  On the other hand, the arguments in this section can be adapted to situations where there exists an element $m \in M_\reg$ such that the functions $\tr^c_{\Theta^c_x}(m)$ are linearly independent. 
The case where $\tX$ has non-isolated fixed points is more mysterious.
\end{remark}

\section{Geometric construction}\label{sec:geometry}
We again fix a conical symplectic singularity $X$ as in Section \ref{sec:cones}, and we now assume that $X$ admits a $(T\times\cs)$-equivariant
projective symplectic resolution $\tX$, which we fix throughout this section.
The odd cohomology of $\tX$ vanishes \cite[Proposition 2.5]{BLPWquant}, thus we have a short exact sequence
\begin{equation}\label{eq:coh-exact}
0\to H^2_{T\times\cs}(pt; \C) \to H^2_{T\times\cs}(\tX; \C) \to H^2(\tX; \C)\to 0,\\
\end{equation}
which we will call the {\bf cohomology exact sequence}.  Given $u\in H^2_{T\times\cs}(\tX; \C)$, let $\bar u$
denote its image in $H^2(\tX; \C)$.

\subsection{Quantum cohomology}
Let $H_2(\tX; \Z)_{\operatorname{free}}$ denote the quotient of $H_2(\tX; \Z)$ by its torsion subgroup.
Let $QH^*_{T\times\cs}(\tX; \C)$ be the equivariant quantum cohomology ring of $\tX$, with the quantum product shifted by the canonical theta characteristic.
The underlying graded vector space of $QH^*_{T\times\cs}(\tX; \C)$ is equal to the tensor product of $H^*_{T\times\cs}(\tX; \C)$ with the completion of the semigroup
ring of the semigroup of effective curve classes in $H_2(\tX; \Z)_{\operatorname{free}}$.  Let $\star$ denote the quantum product
and let $\hbar\in H^*_{T\times\cs}(pt; \C)$ be the weight of the symplectic form.
In \cite[Section 2.3.4]{okounkov2015enumerative}, Okounkov conjectures that there exists a finite set $\Delta_+\subset H_2(\tX; \Z)_{\operatorname{free}}$
and an element $L_\a\in H^{2\dim X}(\tX \times_X \tX; \C)$ for each $\a\in\Delta_+$
such that, for all $u\in H^2_{T \times \C^{\times}}(\tX; \C)$, 
$$u\star \cdot = 
u \cup \cdot + \hbar \sum_{\a\in\Delta_+} \langle \a, \bar u\rangle \frac{q^\a}{1-q^\a} L_\a(\cdot),$$
where $L_\a$ acts via convolution.
We will assume that this conjecture holds. The minimal such subset $\Delta_+$ is called the set of {\bf positive K\"ahler roots},
and the set $\Delta := \Delta_+ \cup -\Delta_+$ is called the set of {\bf K\"ahler roots}.
We will further assume that $\Delta_+$ spans the lattice $H_2(\tX; \Z)_{\operatorname{free}}$.

Let $$F := \C\{q^\a\mid \a\in\N\Delta_+\}\and
\Freg := F\!\left[\textstyle\frac{1}{1-q^{\a}}\;\Big{|}\; \a\in\Delta_+\right].$$
Our assumption that $\Delta_+$ spans $H_2(\tX; \Z)_{\operatorname{free}}$ implies that $\Spec F$ contains the {\bf K\"ahler torus} 
$K := H^2(\tX, \C^\times)$ as a dense open subset. Our assumptions also imply that quantum multiplication by a divisor restricts to an operator on the graded vector space
$$\Qreg := \Freg \otimes H^*_{T\times\cs}(\tX; \C),$$ where $\Freg$ lives in degree zero.

\begin{remark}\label{not quite the ring}
If the quantum cohomology ring $QH^*_{T\times\cs}(\tX; \C)$ is generated by divisors, then our assumption implies that $\Qreg$ is a subring of the quantum
cohomology ring.  In general, however, we do not know that $\Qreg$ is a ring, as we make no assumptions about quantum multiplication by classes in degree greater than two.
\end{remark}

\subsection{The specialized quantum D-module}\label{sec:quantum-connection}
Let $$E := F\otimes \Sym H^2_{T\times\cs}(\tX; \C),$$
with multiplication defined by the formula
$$u\, q^\a = q^\a \big(u + \hbar \langle\a,\bar u\rangle\big)$$
for all $\a\in\N\Delta_+$ and $u\in\Sym H^2_{T\times\cs}(\tX; \C)$.
We also let $\Ereg$ be the Ore localization of $E$ with respect to the multiplicative subset generated by $(1 - q^{\alpha})$ for $\alpha \in \Delta_+$,
which is well-defined by the same argument as in Lemma \ref{orecondition}.
The algebra $\Ereg$ acts in a natural way on $\Qreg = \Freg \otimes H^*_{T\times\cs}(\tX; \C).$
The subring $\Freg \subset \Ereg$ acts by multiplication on the first tensor factor, while an element $u\in H^2_{T\times\cs}(\tX; \C)$
acts by the operator $\hbar \partial_{\bar u} - u \star $ where $\partial_{\bar u}$ is the $\cohtor$-equivariant vector field on $\Spec F$
defined by the equation $\partial_{\bar u}(q^\a) = \langle \a, \bar u \rangle q^\a $.

\begin{remark}\label{quantum D-module}
Just as we defined $R_T^c$ by specializing $H_2(\tX; \C)\subset\Azo$ and localizing from $\Spec S$ to $T$
(Remark \ref{diff-ops}),
we also define $E_K^c$ for any $c\in\ft$ by specializing $\ft^*\subset H^2_{T\times\cs}(\tX; \C)$
and localizing from $\Spec F$ to $\cohtor$.  Then
$E_K^c$ is non-canonically isomorphic to the Rees algebra of differential operators on $\cohtor$.
If we further localize to $\cohtorreg := \cohtor\cap \Spec \Freg$,
we obtain the Rees algebra $E^c_{K_\reg}$ of differential operators on $\cohtorreg$ acting on $\mathcal{O}(\cohtorreg) \otimes H^*_{\cs}(\tX; \C)$.
\end{remark}

\begin{remark}\label{CY}
In our construction, the ring $\C[\hbar]$ plays two {\em a priori} unrelated roles. It is identified both with the subring of $E$ generated by the Rees parameter, and with the ring of equivariant parameters $H^2_{\cs}(pt; \C)$ acting on $H^*_{T \times \cs}(\tX; \C)$.
There is a more general construction in which one does not make this identification. 
Let  $$\mathfrak{E} := F\otimes \Sym H_{T\times\cs}^2(\tX; \C) \otimes \C[z]$$ be the algebra with relations $u\, q^\a = q^\a \big(u + z \langle\a,\bar u\rangle\big)$, containing the central subalgebra $\C[z,\hbar]$; we then have $E \cong \mathfrak{E} / (z - \hbar) \frak{E}$. 
The ring $\mathfrak{E}$ and its localization $\mathfrak{E}_\reg$ act in a natural way on $$\frak{Q}_\reg := \mathcal{O}(\cohtorreg) \otimes H^*_{T\times\cs}(\tX; \C) \otimes \C[z],$$
and we have $\Qreg \cong \frak{Q}_\reg/ (z-\hbar) \frak{Q}_\reg$. The $\mathfrak{E}_\reg$-module $\frak{Q}_\reg$ 
is traditionally called the {\bf quantum D-module}. Thus our construction is a specialization of the quantum D-module, sometimes called the {\bf Calabi-Yau specialization}. This specialization is often quite drastic: in many known cases, the monodromy of the quantum D-module becomes trivial at $z=\hbar$. Although the module $\frak{Q}_\reg$ motivates our definition of $\Qreg$, it will play no further role in this paper.
\end{remark}

\begin{remark}\label{Q0}
The advantage of working over $\Spec \Freg$ rather than over $\cohtorreg$
is that it makes sense to set $q$ equal to zero.  The specialization $Q_0 := \C_0\otimes_{\Freg} \Qreg$ is a module
over $\Sym H^2_{T\times\cs}(\tX; \C)$, and it is canonically isomorphic to $H_{T \times \cs}^*(\tX;\C) $.
\end{remark}
\begin{example}\label{sl2example3}
Continuing from Examples \ref{sl2example1} and \ref{sl2example2}, suppose that $X$ is the Kleinian singularity of type $A_1$,
in which case $\tX \cong T^*\mathbb{P}^1$.
We may choose a basis $\hbar, u_1, u_2$ for $H^2_{T\times\cs}(\tX; \C)$ such that
$\bar u_1 = \bar u_2$ and the classical cohomology ring is
$$H^*_{T\times\cs}(\tX; \C) \cong \C[u_1,u_2,\hbar]\Big{/}\langle u_1u_2\rangle.$$
In quantum cohomology, we have $$u_1 \star u_2 = \frac{\hbar q}{1-q}(\hbar + u_1 + u_2) = q (\hbar + u_1) \star (\hbar + u_2).$$
This implies that
$$QH^*_{T\times\cs}(\tX; \C) \cong \C[u_1,u_2,\hbar][[q]]\Big{/}\big\langle u_1\star u_2 - q (\hbar + u_1) \star (\hbar + u_2)\big\rangle$$
and
$$\Qreg \cong \C\!\left[u_1,u_2,\hbar,q, (1-q)^{-1}\right]\Big{/}\big\langle u_1 \star u_2 - q (\hbar + u_1) \star (\hbar + u_2)\big\rangle.$$
This is a module over $\Ereg$, which is generated over $\C[\hbar]$ by $u_1$, $u_2$, $q$, and $(1-q)^{-1}$, with $[u_1,u_2] = 0$ and $u_iq = q(u_i+\hbar)$.
The element $q$ acts on $\Qreg$ by scalar multiplication and $u_i$ acts by $\hbar\partial - u_i\star$, where $\partial$ is the vector field given by $\partial(q) = q$.
\end{example}

\section{Duality}\label{sec:duality}
In this section we formulate our main conjecture, which says that the modules constructed in Sections \ref{sec:algebra} and \ref{sec:geometry}
are swapped under symplectic duality.

\subsection{Symplectic duality}\label{sec:sd}
Let $X^!$ be {\bf symplectic dual} to $X$ in the sense of \cite[Section 10.3]{BLPWgco}.
We assume that $X^!$ admits a symplectic resolution $\tX^!$, but we make no such assumption about $X$.
We will not review the full definition of symplectic duality here, but rather focus
on certain manifestations of this relationship and specific examples of dual pairs, which we list below.
Our notational convention will be to denote everything related to $X^!$ with a superscript.  For example, $T^! $ is the maximal torus of $ \Aut(X^!)$, and so on.
\begin{itemize}
\item
The most important aspect of the relationship between $X$ and $X^!$ is that
the quantization exact sequence \eqref{eq:quant-exact} for $X$ is isomorphic to the cohomology exact sequence \eqref{eq:coh-exact}
for $X^!$.  That is, we have isomorphisms $ H_2(\tXreg; \C) \cong (\ft^!)^*$, $\Azo \cong H^2_{T^!\times\cs}(\tX^!; \C) $, and
$\ft \cong H^2(\tX^!; \C)$, compatible with the maps.\footnote{The existence of such isomorphisms appears in \cite[Definition 10.15]{BLPWgco}.
A choice of isomorphism $\Azo \cong H^2_{T^!\times\cs}(\tX^!; \C)$ was not previously considered to be part of the data of symplectic
duality, but the examples that we consider here all come with a natural such choice.}  Moreover, we assume that this last isomorphism 
restricts to a lattice isomorphism $ \ft^*_\Z \cong H_2(\tX^!; \Z) $, which induces an isomorphism $T \cong K^! $. 
\item
In Section \ref{sec:ring}, we had to choose a generic vector $\xi\in\ft_\Z\subset\ft_\R\cong H^2(\tX^!; \R)$
that does not vanish on any of the equivariant roots of $X$.
It is expected that the first Chern class of any ample line bundle on $\tX^!$ will have this property,
and that with this choice the positive equivariant roots for $X$ will coincide with the positive K\"ahler roots for $X^!$ \cite[Section 3.1.8]{okounkov2015enumerative}
(see also the coincidence of the twisting and shuffling arrangements in \cite[Definition 10.1]{BLPWgco}).
We will assume that this is the case.  In particular, this implies that the rings $S$ and $F^!$ are canonically identified, leading to an isomorphism of toric varieties $\Spec S \cong  \Spec F^! $ extending the above isomorphism of tori.  Moreover, this implies a graded ring isomorphism $$R= S\otimes\Sym\Azo\cong F^!\otimes H^2_{T^!\times\cs}(\tX^!; \C) = E^!.$$
\end{itemize}

The main examples of dual pairs that we will consider in this paper are hypertoric varieties, which are dual to other hypertoric varieties,
and the nilpotent cone in $\fg^*$ for a semisimple Lie algebra $\fg$, which is dual to the nilpotent cone in $(\fg^!)^*$, where $\fg^!$ is the Langlands dual
of $\fg$.  Given an inclusion of groups $G\subset\tilde{G}$ and a representation $V$ of $\tilde{G}$, one can construct the Coulomb and Higgs branches
of the associated gauge theory; when they are both conical symplectic singularities, they are expected to be symplectic dual.  This class of examples
includes hypertoric varieties (the case where $\tilde{G}$ is abelian) as well as the nilpotent cone in $\sl_n^*$.

\subsection{Main conjecture} \label{sec:mainconj}
Let $X$ and $X^!$ be symplectic dual.

\begin{conjecture}\label{pi}
There is an isomorphism $\Mreg\cong \Qreg^!$ of graded modules over $\Rreg\cong\Ereg^!$
taking $1\in\Mreg$ to $1\in\Qreg^!$.
\end{conjecture}

We will prove that this conjecture holds for hypertoric varieties (Theorem \ref{hypertoric-pi}) and 
for the Springer resolution (Theorem \ref{Springer-pi}).

\begin{remark}
Choose $ c \in H^2(\tXreg; \C) \cong \ft^! $ and choose a splitting for the quantization exact sequence \eqref{eq:quant-exact} for $X$,
which is isomorphic by assumption to the cohomology exact sequence \eqref{eq:coh-exact} for $X^!$.  As in Remarks \ref{diff-ops}
and \ref{quantum D-module}, we obtain an isomorphism between $R^c_{T_\reg} \cong (E_{K_\reg}^{!})^c $ and the Rees algebra of differential operators on $ \Treg \cong \cohtorreg^!$.  Thus Conjecture \ref{pi} becomes an isomorphism between modules over this ring of differential operators.
\end{remark}

\begin{remark}\label{Nak-Hik}
Proposition \ref{M0}, Remark \ref{Q0}, and Conjecture \ref{pi} together imply that we have an isomorphism 
$$\HH_0(B(\scrA))\cong M_0 \cong Q_0^! \cong H^*_{T^!\times\cs}(\tX^!; \C)$$
of graded modules over $\Sym\Azo \cong \Sym H^2_{T^!\times\cs}(\tX^!; \C)$.
If we further assume that $B(\scrA)$ is commutative, which is true in all known examples, then we obtain an isomorphism
$B(\scrA)\cong H^*_{T^!\times\cs}(\tX^!; \C)$.
This is closely related to Nakajima's {\bf equivariant Hikita conjecture} \cite[Conjecture 8.9]{KTWWY},
which states that $B(\scrA)$ and $H^*_{T^!\times\cs}(\tX^!; \C)$ are isomorphic as graded rings.
In cases where $H^*_{T^!\times\cs}(\tX^!; \C)$ is generated in degree 2, such as the hypertoric and Springer cases considered
in this paper, the two statements are in fact equivalent.
\end{remark}

\begin{remark}
We know that $\Qreg^! = \Freg^! \otimes H^*_{T^!\times\cs}(\tX^!; \C) $ is finitely generated over the ring
$\Freg^!\otimes H^*_{T^!\times\cs}(pt; \C)$, which is isomorphic to $ \Sreg \otimes \Sym H_2(\tXreg; \C) \otimes \C[\hbar]$.
Thus Conjecture \ref{pi} would imply that $\Mreg$ is finitely generated over the same ring.
Assuming the equivariant Hikita conjecture $M_0 \cong Q_0^!$ (Remark \ref{Nak-Hik}), we know that
$ M_0$ is finitely generated over $  H^*_{T^! \times \cs}(pt; \C) \cong \Sym H_2(\tXreg; \C) \otimes \C[\hbar] $.
Proposition \ref{finitely generated} then implies that $\Mreg $ is indeed finitely generated over $ \Sreg \otimes \Sym H_2(\tXreg; \C) \otimes \C[\hbar]   $.
Thus Proposition \ref{finitely generated} may be regarded as supporting evidence for Conjecture \ref{pi}.
\end{remark}

\subsection{Weyl groups}\label{sec:weyl}
The {\bf Namikawa Weyl group} \cite{NamiaffII} of $X$ is a finite group equipped with a faithful action on
$ H_2(\tXreg; \C) $.  As part of the package of symplectic duality, $ W $ is identified with the Weyl group of the reductive group $ \Aut(X^!) $
in a manner compatible with the actions on $ H_2(\tXreg; \C) \cong (\ft^!)^*. $
These actions
extend to actions on $ \Azo $ and $ H^2_{T^! \times \cs}(\tX^!; \C) $ and then to the rings $ \Rreg $ and $ \Ereg^! $ (acting trivially on $ \Sreg $ and $ \Freg$).
Moreover, $W$ acts compatibly on the modules $ \Mreg $ and $ \Qreg^! $, 
and it is natural to expect the isomorphism of Conjecture \ref{pi} to be $W$-equivariant.
In particular, this would imply that the $W$-invariant parts are isomorphic, as we conjecture below.

\begin{conjecture}\label{piW}
There is an isomorphism $\Mreg^W\cong (\Qreg^!)^W$ of graded modules over $\Rreg^W\cong(\Ereg^!)^W$
taking $1\in\Mreg^W$ to $1\in(\Qreg^!)^W$.
\end{conjecture}


Let us examine the objects appearing in Conjecture \ref{piW} for future use.  On the algebraic side, we have
$ \scrA \cong \scrA^W \otimes_{(\Sym H_2(\tXreg; \C))^W} \Sym H_2(\tXreg; \C) $
\cite[Proposition 3.5]{Losev-orbit} and therefore
$$ J \cong J(\scrA^W) \otimes_{(\Sym H_2(\tXreg; \C))^W} \Sym H_2(\tXreg; \C), $$
where
$$ J(\scrA^W) =  \sum_{\la\in\N\Sigma_+}S[\hbar]\cdot\left\{ 1\otimes a b - q^\lambda \otimes b a \mid a \in \scrA^W_\lambda, b \in \scrA^W_{-\lambda}\right\}.$$
This implies that
$$ \Mreg^W = (\Sreg \otimes \scrA_0 / J)^W \cong \Sreg \otimes \scrA_0^W / J(\scrA^W).$$
In other words, $ \Mreg^W $ is obtained by applying our construction of the module $ \Mreg $ to the invariant algebra $ \scrA^W $. On the other hand, since $ \scrA $ is obtained from $ \scrA^W $ by extension of scalars, we have
\begin{equation} \label{eq:MfromMreg} \Mreg \cong \Mreg^W \otimes_{(\Sym H_2(\tXreg; \C))^W} \Sym H_2(\tXreg; \C)
\end{equation}
as $ \Rreg$-modules.

On the geometric side, we have
$$
 (\Qreg^!)^W \cong \Freg \otimes H^*_{T^! \times \cs}(\tX^!; \C)^W \cong \Freg \otimes H^*_{\Aut(\tX^!) \times \cs}(\tX^!; \C).
$$
Moreover, we have an isomorphism $$H^*_{T^! \times \cs}(\tX^!; \C) \cong  H^*_{\Aut(\tX^!) \times \cs}(\tX^!; \C) \otimes_{H^*_{\Aut(\tX^!)}(pt)} H^*_{T^!}(pt).$$ This isomorphism is compatible with quantum multiplication by divisors, thus 
\begin{equation} \label{eq:QfromQreg}
\Qreg^! \cong (\Qreg^!)^W\otimes_{H^*_{\Aut(\tX^!)}(pt)} H^*_{T^!}(pt) 
\end{equation}
as $\Ereg^!$-modules.
Comparing (\ref{eq:MfromMreg}) and (\ref{eq:QfromQreg}) shows that Conjecture \ref{piW} implies Conjecture \ref{pi}.

\subsection{Beyond the regular locus} \label{thegreatbeyond}
In Section \ref{sec:algebra}, we defined a module $M$ over $R$, and then localized to obtain a module $\Mreg$ over $\Rreg$.
In Section \ref{sec:geometry}, however, we did not define a module $Q$ over $E$ that localizes to the module $\Qreg$ over $\Ereg$.
To this end, we now define $Q$ to be the $E$-submodule of $\Qreg$ generated by the vector subspace
$$1 \otimes H^*_{T \times \cs}(\tX; \C) \subset \Freg \otimes H^*_{T \times \cs}(\tX; \C) = \Qreg.$$
By definition, $Q$ is a subspace of $\Qreg$.  Note that the situation with $M$ and $\Mreg$ is less clear; there is a natural
map from $M$ to $\Mreg$, but this map could {\em a priori} fail to be injective if $M$ has nontrivial $S$-torsion.
Nonetheless, for any particular symplectic dual pair, it is natural to ask the following question.

\begin{question}\label{MQ}
Do we have an isomorphism $M\cong Q^!$ of graded modules over $R\cong E^!$?
\end{question}

In the hypertoric case, Question \ref{MQ} is equivalent to the question of whether or not $M$ is torsion-free as a module over $S$.
We believe that the answer is yes, and this problem will be addressed in a future work.  In the case of the Springer resolution,
we believe that the answer is yes in type A and no in other types.  It is interesting to note that the Springer resolution
can be realized as a quiver variety
only in type A, so it is possible that the answer to Question \ref{MQ} is yes for all
dual pairs arising from the Higgs/Coulomb construction associated with a linear representation of a group.

\begin{remark}\label{whoa}
In a case where the answer to Question \ref{MQ} is yes, we obtain an isomorphism $(Q^!_1)^{\hbar =1} \cong M_1^{\hbar = 1}\cong \HH_0(\scrA^{\hbar = 1})$ by Proposition \ref{M1}.
The second and third authors have conjectured that
$(Q^!_1)^{\hbar =1} \cong I\! H_{T^!\times\cs}^*(X^!; \C)^{\hbar=1}$ \cite[Conjecture 2.5]{McBP}.
On the other hand, the third author has conjectured that
$\HH_0(\scrA^{\hbar = 1})$ is isomorphic to
$I\! H_{T^!\times\cs}^*(X^!; \C)^{\hbar=1}$ \cite[Conjecture 3.6]{Pro12}.
Thus a positive answer to Question \ref{MQ} would make these two conjectures equivalent to each other.
\end{remark}


\section{The hypertoric case}\label{sec:hypertoric}
The purpose of this section is to prove Conjecture \ref{pi} for hypertoric varieties.  
We begin with a review of quantization, quantum cohomology, and symplectic duality in the hypertoric setting.

\subsection{Hypertoric varieties}\label{sec:basics}
Fix an algebraic torus $T$, a positive integer $n$, and an $n$-tuple $\gamma\in\left(\ft_\Z\right)^n$ satisfying the following conditions:
\begin{itemize}
\item[(1)] For all $i$, $\gamma_i$ is nonzero, primitive, and contained in the span of $\{\gamma_j\mid j\neq i\}$.
\item[(2)] The semigroup $\mathbb{N}\{\gamma_1,\ldots,\gamma_n\}$ spanned by $\gamma$ is equal to $\ft_\Z$.
\item[(3)] For any subset $S\subset [n]$, if $\{\gamma_i\mid i\in S\}$ is a $\C$-basis for $\ft$, then it is also a $\Z$-basis for $\ft_\Z$.
\end{itemize}
These cocharacters define a surjective map from $\cs^n$ to $T$, and we let $K$ denote the kernel of this map.
Consider the action of $\cs^n$ on $\C^n$ by coordinate-wise scalar multiplication.  This induces a symplectic action
on $T^*\C^n$, and the subtorus $K\subset\cs^n$
acts with moment map $\mu:T^*\C^n\to\fk^*$, where $\mu(q,p)$ is the projection of $(q_1p_1,\ldots,q_np_n)\in \C^n \cong \Lie\left(\cs^n\right)^*$
to $\fk^*$.
Let $$X := \mu^{-1}(0)/\!\!/_{\! 0}\, K = \Spec \cO\!\left(\mu^{-1}(0)\right)^K$$ denote the
the affine symplectic quotient of $T^*\C^n$ by $K$.
The inverse scaling action of $\cs$ on $T^*\C^n$
induces an action on $X$ for which $\mathcal{O}(X)$ is non-negatively graded with only constants in degree 0
and the natural Poisson bracket has weight -2.
The fact that the degree 1 part of $\mathcal{O}(X)$ is trivial follows from the last part of item (1).
Fix a character $\theta\in\fk^*_\Z$ and consider the GIT quotient 
$$\tX := \mu^{-1}(0)/\!\!/_{\!\theta}\, K = \Proj \left(\cO\!\left(\mu^{-1}(0)\right)\otimes\C[t]\right)^K,$$
where $\Proj$ is taken with respect to the grading given by powers of $t$ and $K$ acts on $t$ with weight $\theta$.
We will assume that $\theta$ is chosen generically; this, along with item (3), implies
that $\tX$ is smooth, and therefore a symplectic resolution of $X$
\cite[Theorem 3.2]{BD}.  
The symplectic action of $\cs^n$ on $T^*\C^n$ descends to an action of $T$ on $X$ and $\tX$ commuting with the conical action of $\cs$. This realizes $T$ as a maximal torus of $\Aut(X)$.
We have canonical isomorphisms
$$\scrX \cong T^*\C^n/\!\!/_{\!0}\, K = \Spec \cO(T^*\C^n)^K
\and
\tilde\scrX \cong T^*\C^n/\!\!/_{\!\theta}\, K.$$

\subsection{The hypertoric enveloping algebra}\label{sec:hea}
Let $D$ be the $\C[\hbar]$-algebra generated
by degree 1 elements $z_1,\ldots,z_n$ and $w_1,\ldots,w_n$ with $[z_i,z_j] = 0 = [w_i,w_j]$
and $[w_i, z_j] = \delta_{ij}\hbar$.  Then the {\bf hypertoric enveloping algebra}
$\scrA := D^K$ is the canonical quantization of $\scrX$.
Let $a_i := z_iw_i$, so that $\scrA_0 = \C[a_1,\ldots,a_n,\hbar]$.
Note that, in the hypertoric case, the canonical ring homomorphism $\Sym\Azo\to\scrA_0$ is an isomorphism.

For all $\la\in\ft^*_\Z$, let $\la_i := \langle\la,\gamma_i\rangle\in\Z$.
The map $\ft^*_\Z\hookrightarrow\Z^n$ taking $\la$ to $(\la_1,\ldots,\la_n)$ coincides with
the inclusion of character lattices induced by the surjection from $\cs^n$ to $T$.
For all $\la$, there is a unique expression of the form $\la = \la_+-\la_-$ where $\la_+,\la_-\in\mathbb{N}^n$ have disjoint support.
Note that $\la_+$ and $\la_-$ will typically not live in the sublattice $\ft^*_\Z\subset\Z^n$.
For all $\la\in\ft^*_\Z$, let $m^\la := z^{\la_+}w^{\la_-}\in\scrA_\la$.
Then we have $\scrA_\la = \scrA_0 m^\la$.

For any $k\in\Z$, let \begin{equation}\label{aik}[a_i]^k := \begin{cases} \text{1 if $k=0$}\\ \text{$a_i(a_i-\hbar)\cdots(a_i-(k-1)\hbar)$ if $k> 0$}\\
\text{$(a_i+\hbar)(a_i+2\hbar)\cdots(a_i-k\hbar)$ if $k< 0$.}\end{cases}\end{equation}
Let $\la_i := \langle\la,\gamma_i\rangle$ be the $i^\text{th}$ coordinate of $\la\in\Z^n$.  Then we have the following description of the product
structure of $\scrA$ \cite[Section 3.2]{JustinThesis}:
\begin{equation}\label{hypertoric product}
[a_i, m^\la] = \la_i \hbar m^\la\and
m^\la m^\mu = \left(\prod_{\substack{\la_i\mu_i<0\\ |\la_i|\leq |\mu_i|}} [a_i]^{\la_i}\right) m^{\la+\mu}
\left(\prod_{\substack{\la_i\mu_i<0\\ |\la_i|> |\mu_i|}} [a_i]^{-\mu_i}\right).
\end{equation}

\subsection{Cocircuits and equivariant roots}
A nonzero primitive element of $\ft^*_\Z$ with minimal support is called a {\bf cocircuit}.
Condition (3) in Section \ref{sec:basics}
implies that, for any cocircuit $\la$, we have $\la_i\in\{-1,0,1\}$ for all $i$.

\begin{lemma}\label{basis}
For all $\la\in\ft^*_\Z$, there exist cocircuits $\mu^1,\ldots,\mu^m$, each supported on a subset of $\supp(\la)$, such that
\begin{equation}\label{sum}
\la = \mu^1 + \cdots + \mu^m.\end{equation}
\end{lemma}

\begin{proof}
We will proceed by induction on the support of $\la$.  If $\la=0$, we are done.
If not, choose a cocircuit $\mu$ such that $\supp(\mu)\subset \supp(\la)$, and choose an element $i\in\supp(\mu)$.
Then $\la - (\la_i/\mu_i) \mu$ has support contained in $\supp(\la)\smallsetminus\{i\}$.  Since $\mu$ is a cocircuit, $\mu_i = \pm 1$,
so $\la_i/\mu_i$ is an integer.
\end{proof}

We call Equation \eqref{sum} {\bf cancellation free} if $\mu^k_i\mu^l_i\geq 0$ for all $i,k,l$.

\begin{lemma}\label{cancellation free}
For any $\la\in\ft^*_\Z$ we may choose $\mu^1,\ldots,\mu^m$ as in Lemma \ref{basis} such that
Equation \eqref{sum} is cancellation free.
\end{lemma}

\begin{proof}
We will again proceed by induction on the support of $\la$.  If the support of $\la$ is minimal, then $\la$ is a multiple of a cocircuit, and we are done.
Otherwise, choose cocircuits $\mu^1,\ldots,\mu^m$ such that $\la = \mu^1+\cdots+\mu^m$ and $\supp(\mu^k)\subset\supp(\la)$
for all $k$, which we can do by Lemma \ref{basis}.
Since each $\mu^k$ is a cocircuit, we have $\mu^k_i\mu^l_i\in\{-1,0,1\}$ for all $i,k,l$.

Let $i$ be the first coordinate such that there exist $k,l$ with $\mu^k_i\mu^l_i = -1$.
Let $d$ be the minimum of $|\{k\mid \mu^k_i = 1\}|$ and $|\{k\mid \mu^k_i = -1\}|$;
we call this the {\bf degree of cancellation} in the $i^\text{th}$ coordinate.
We will produce a new expression for $\la$ that has a degree of cancellation of $d-1$ in the $i^\text{th}$ coordinate
and still has no cancellation in the $j^\text{th}$ coordinate for $j<i$.  By a second induction, this time on the index $i$, this will imply that we can obtain
a cancellation free expression for $\la$.

Choose $k$ and $l$ such that $\mu_i^k\mu_i^l=-1$.  This means that $\mu_i^k+\mu_i^l=0$, so we have
$\supp(\mu^k+\mu^l)\subset\supp(\la)\smallsetminus\{i\}$.  By our (first) inductive hypothesis, there exist cocircuits
$\nu^1,\ldots,\nu^s$ such that $\supp(\nu^t)\subset\supp(\mu^k+\mu^l)$ for all $1\leq t\leq s$, 
$\mu^k+\mu^l = \nu^1+\cdots+\nu^s$, and this sum is cancellation free.
Then we have $$\la = \mu^1 + \cdots + \widehat{\mu^k} + \cdots +\widehat{\mu^l} + \cdots + \mu^m + \nu^1 + \cdots + \nu^s,$$
with the support of each of the cocircuits on the right-hand side contained in the support of $\la$, and with a degree of cancellation
of $d-1$ in the $i^\text{th}$ coordinate.  Thus it remains only to show that this expression has no cancellation in the $j^\text{th}$
coordinate when $j<i$.

Assume that there is cancellation in the $j^\text{th}$ coordinate for some $j<i$; this means that we have indices $q$ and $p$ such
that $\mu^q_j\nu^p_j = -1$.  Assume further that $\mu^q_j = 1$ and $\nu^p_j = -1$ (the opposite case is identical).
Since $\mu^q_j = 1$ and the sum $\mu^1+\cdots+\mu^m$ has no cancellation in the $j^\text{th}$ coordinate,
we have $\mu^k_j,\mu^l_j\in\{0,1\}$, and in particular the $j^\text{th}$ coordinate of $\mu^k+\mu^l = \nu^1+\cdots+\nu^s$
is non-negative.  Since the sum on the right is cancellation-free, this implies that $\nu^p_j\in\{0,1\}$, which contradicts
our assumption.
\end{proof}

\begin{proposition}\label{toric eq roots}
The equivariant roots of $X$ are precisely the cocircuits.
\end{proposition}

\begin{proof}
If $\la$ is a cocircuit, it is clear that $m^\la\in\scrA^+\smallsetminus \left(\scrA^+\right)^2$, so $\la$ is an equivariant root.
By definition, $0$ is not an equivariant root.
Now suppose that $\la\neq 0$ is not a cocircuit.  By Lemma \ref{cancellation free}, we may write $\la = \mu + \nu$, where
$\mu,\nu\in\ft^*_\Z$ and $\mu_i\nu_i \geq 0$ for all $i$.  Then Equation \eqref{hypertoric product}
tells us that $m^\la = m^\mu m^\nu$, and therefore $\scrA_\la = \scrA_0 m^\la\subset \left(\scrA^+\right)^2$, so $\la$
is not an equivariant root.
\end{proof}

Fix an element $\xi\in\ft_\Z$ such that $\langle\xi,\la\rangle\neq 0$ for every cocircuit $\la$.  We call a cocircuit {\bf positive}
if $\langle\xi,\la\rangle > 0$.
By Proposition \ref{toric eq roots}, equivariant roots are the same as cocircuits, so $\Sigma_+$
is equal to the set of positive cocircuits.

\begin{example}\label{counterexample}
It is tempting to think that every element of $\N\Sigma_+$ can be written as a cancellation-free sum of positive cocircuits.
Unfortunately, this is not the case, as we illustrate here.
Let $T := \cs^3$.  Let $\{e_1,e_2,e_3\}$ be the coordinate basis for the character lattice $\ft^*_\Z$,
and let $\{x_1,x_2,x_3\}$ be the dual basis for the cocharacter lattice $\ft_\Z$.
Let $\gamma_1 := x_1$, $\gamma_2 := x_2$, $\gamma_3 := x_3$,
$\gamma_4 := x_1-x_3$, and $\gamma_5 := x_2-x_3$.  Let $\xi := x_1 - 3x_2 + x_3$.
Then we have $\Sigma_+ = \{e_1, -e_2, e_3, -e_2-e_3, e_1+e_3, -e_1-e_2-e_3\}$.
Let $\la :=  -e_1 - e_2$.
We can write $\la$ as the sum of two positive cocircuits ($e_3$ and $-e_1-e_2-e_3$) or as the cancellation-free
sum of two cocircuits ($-e_1$ and $-e_2$), but not as the cancellation-free sum of two positive cocircuits.
\end{example}

\subsection{The $\mathbf{R}$-module $\mathbf{M}$} \label{modulesec}
Let $S := \C\{q^\la\mid\la\in\N\Sigma_+\}$ as in Section \ref{sec:ring}.
Since $\Sym\Azo \cong \scrA_0$, we have $$R := S\otimes \Sym\Azo \cong S\otimes \scrA_0,$$
and the left $R$-module $S\otimes \scrA_0$ is simply the left regular module.
Let $R_T := \mathcal{O}(T) \otimes_S R$.  Since $R$ is free as an $S$-module, the natural map from $R$ to $R_T$ is an inclusion,
thus we can freely work inside of $R_T$ when doing calculations in $R$.

For any $f(a)\in\scrA_0$ and $\la\in\ft^*_\Z$, let $$f_\la(a) := f(a_1+\la_1\hbar,\ldots,a_n+\la_n\hbar),$$ so that we have
$$f(a) m^\la  = m^\la f_{\la}(a) \in\scrA \and
f(a) q^\la = q^\la f_\la(a) \in R_T.$$
For any $\la\in\ft^*_\Z$, let
\begin{equation}\label{ala}[a]^\la := \prod_i [a_i]^{\la_i} = m^\la m^{-\la}\in \scrA_0.\end{equation}
We observe that, for all $\la\in\ft^*_\Z$, we have $[a]^\la_\la = [a]^{-\la}$.

By definition, $M$ is the left $R$-module $R/J$, where $J$ is the left ideal generated by elements of the form
\begin{eqnarray*}f(a)m^\la g(a)m^{-\la} - q^\lambda g(a)m^{-\la} f(a)m^\la
&=& f(a) g_{-\la}(a) m^\la m^{-\la} - q^{\la} g(a) f_\la(a) m^{-\la} m^\la\\
&=& f(a) g_{-\la}(a) [a]^\la - f(a) g_{-\la}(a) q^\la [a]^{-\la}\\
&=& f(a) g_{-\la}(a) [a]^\la (1-q^\la)
\end{eqnarray*}
for $f(a), g(a) \in\scrA_0$ and $\la\in\N\Sigma_+$.
Therefore, if we define
$$r(\la) := [a]^{\la} (1-q^\la),$$
then
$$J = R\cdot \{ r(\la) \mid \la\in\N\Sigma_+ \}.$$

The rest of this section will be devoted to proving Proposition \ref{JJ'}, which says that $J$ is in fact generated by those classes
$r(\la)$ for $\la$ a positive cocircuit (rather than a sum of positive cocircuits).
The following three lemmas are completely straightforward, so we omit their proofs.

\begin{lemma}\label{la-minus-la}
For any $\la\in\ft^*_\Z$, $r(\la) = -q^\la r(-\la)\in R_T$.
\end{lemma}


\begin{lemma}\label{key equation}
If $\la = \mu+\nu$, then $r(\la) = q^\nu [a]^\la_\nu (1-q^\mu) + [a]^\la (1-q^\nu)\in R_T$.
\end{lemma}

\begin{lemma}\label{cf}
If $\mu,\nu\in\ft^*_\Z$ and $\mu_i\nu_i\geq 0$ for all $i$, then there exist $f(a), g(a)\in \scrA_0$ such that
$$[a]^{\mu+\nu}_\nu = f(a)[a]^\mu\and [a]^{\mu+\nu} = g(a)[a]^\nu.$$
\end{lemma}

\begin{proposition}\label{JJ'}
We have $$J = R\cdot \{ r(\la) \mid \la\in\Sigma_+ \}.$$
\end{proposition}

\begin{proof}
Let $J' := R\cdot \{ r(\la) \mid \la\in\Sigma_+ \}.$  We need to show that, if $\la\in\N\Sigma_+$, then $r(\la)\in J'$.
By Lemma \ref{cancellation free}, we can choose cocircuits $\mu^1,\ldots,\mu^m$ (not necessarily positive) such that $\la=\mu^1+\cdots+\mu^m$
and this sum is cancellation-free.
We will prove that $r(\la)\in J'$ by induction on $m$.  The base case $m=0$ follows from the fact that $r(0)=0$.

Let us first assume that $\mu^i\in\Sigma_+$ for all $i$.  Let $\mu = \mu^m$ and $\nu = \mu^1+\cdots+\mu^{m-1}$.
By Lemma \ref{key equation}, we have $$r(\la) = q^\nu [a]^\la_\nu (1-q^\mu) + [a]^\la (1-q^\nu).$$
By Lemma \ref{cf}, there are elements $f(a), g(a)\in \scrA_0$ such that
\begin{eqnarray*}r(\la)
&=& q^\nu f(a)[a]^\mu (1-q^\mu) + g(a)[a]^\nu (1-q^\nu)\\
&=& q^\nu f(a) r(\mu) + g(a) r(\nu).
\end{eqnarray*}
Our inductive hypothesis tells us that $r(\mu)\in J'$, so we are done.

Next, assume that there is at least one index $i$ for which $-\mu_i\in\Sigma_+$.  After reordering, we may assume that $i=m$.
Once again, let $\mu = \mu^m$ and $\nu = \mu^1+\cdots+\mu^{m-1}$.
By the same reasoning as above, there are elements $f(a), g(a)\in \scrA_0$ such that
$$r(\la) = q^\nu f(a) r(\mu) + g(a) r(\nu).$$
By Lemma \ref{la-minus-la}, we have $r(\mu) = -q^\mu r(-\mu)$, so
\begin{eqnarray*} r(\la) &=& -q^{\nu} f(a) q^\mu r(-\mu) + g(a) r(\nu)\\
&=& -q^{\nu+\mu} f_{\mu}(a) r(-\mu) + g(a) r(\nu)\\
&=& -q^{\la} f_{\mu}(a) r(-\mu) + g(a) r(\nu).
\end{eqnarray*}
Since $-\mu\in\Sigma_+$, $r(-\mu)\in J'$.  Since $\nu = \la + (-\mu) \in \N\Sigma_+$ and $\nu$ is a cancellation-free
sum of $m-1$ cocircuits, our inductive hypothesis tells us that $r(\nu)\in J'$.
Thus $r(\la)\in J'$, as desired.
\end{proof}

\begin{remark}
If we knew that $\la\in\N\Sigma_+$ could be written as a cancellation-free sum of positive cocircuits, then the last paragraph
of the proof of Proposition \ref{JJ'} would be unnecessary.  However, this is not always the case, as we saw in Example \ref{counterexample}.
In that particular example, we have
$r(\la) = -q^\la a_2a_5 r(e_1) + (a_1+\hbar)(a_4+\hbar) r(-e_2)$.
\end{remark}

\subsection{The $\mathbf{\Ereg}$-module $\mathbf{\Qreg}$}
Recall that
$$E := F\otimes \Sym H^2_{T\times\cs}(\tX; \C)$$
and $\Ereg$ is the Ore localization of $E$ obtained by inverting $(1-q^{\alpha})$ for all $\alpha \in \Delta_+$. 

\begin{proposition}\label{classical}
The $\Ereg$-module $\Qreg$ is cyclic, generated by the class $1\otimes 1\in \Qreg$.
\end{proposition}

\begin{proof}
Given a $k$-tuple $u = (u_1, \ldots, u_k) \in H_{T\times \cs}^{\bullet}(\tX;\C)^k$,
we define $u_*\in \Qreg$ to be the quantum product of the entries, and we define $u_\cup \in Q$
to be the tensor product of $1\in F$ with the classical product of the entries.  We call $u$ {\bf classical} if $u_* = u_\cup$. It is easy to see that,
if $u$ is a classical tuple, then  $u_1 u_2 \cdots u_k \cdot 1 \otimes  1 = \pm 1 \otimes u_\cup$ in $Q_\reg$.
The cohomology ring
$H_{T\times \cs}^{\bullet}(\tX;\C)$ is spanned over $H^{\bullet}_{T \times \cs}(pt; \C)$
by products of classical tuples of divisors \cite[Corollary 3.3 and Lemma 3.4]{McBP},
therefore $\Qreg$ is generated by $1\otimes 1$.
\end{proof}

It follows from Proposition \ref{classical} that $\Qreg$ is isomorphic as an $\Ereg$-module to the quotient of the regular module $\Ereg$
by some left ideal, namely the annihilator of $1\otimes 1$.  Our goal is now to compute that ideal.

By definition, we have $K\subset\cs^n$ and therefore $\fk_\Z\subset\Z^n$.  For any $\a\in\fk_\Z$ and $i\leq n$, we define $\a_i\in\Z$
to be its $i^\text{th}$ coordinate.
A nonzero primitive element of $\fk_\Z$ with minimal support is called a {\bf circuit}.  Each circuit pairs nontrivially with
the element $\theta\in\fk_\Z^*$, and we call a circuit $\a$ {\bf positive} if $\langle\a,\theta\rangle>0$.

Let $\chi_1,\ldots,\chi_n$ be the coordinate basis for $\Z^n = \Hom(\mathbb{G}_m^n, \mathbb{G}_m)$. The equivariant Kirwan map is an isomorphism
from $\Z[\chi_1,\ldots,\chi_n,\hbar/2]$ to $H^2_{T\times\cs}(\tX; \Z)$, which takes a character of $\cs^n \times \cs$ to the $T \times \cs$ equivariant chern class of the associated line bundle on the quotient. We let $u_i$ be the image of $\chi_i - \hbar/2$.\footnote{The
class $\chi_i - \hbar/2$ is the weight of the normal bundle to the hyperplane $\{z_i = 0 \}$ in $T^*\C^n$, and $u_i$ is therefore represented 
by the Borel-Moore cycle given by the image of this hyperplane in $\tX$.} Setting $\hbar/2$ equal to zero and passing to the quotient
$\fk_\Z^*$ of $\Z^n$, we obtain the ordinary Kirwan map, which is an isomorphism from
$\fk^*_{\Z}$ to $H^2(\tX, \Z)$ taking $\theta$ to the first Chern class of an ample line bundle.
The dual isomorphism $\fk_\Z \cong H_2(\tX, \Z)$
takes the circuits bijectively to the K\"ahler roots of $\tX$ and the positive circuits to the set $\Delta_+$ of positive K\"ahler roots
\cite[Theorem 4.2]{McBS}.

For each positive circuit $\a\in\Delta_+$, let $$s(\a) := [u]^\a(1-q^\a)\in \Ereg,$$
where $[u]^\a\in \Sym H^2_{T\times\cs}(\tX; \C)$ is defined in a manner analogous to the definition of $[a]^{\la}$ in Equations
\eqref{aik} and \eqref{ala}.  Let $$\Ireg := \Ereg\cdot \{s(\a)\mid \a\in \Delta_+\}.$$
We will prove that $\Ireg\subset\Ereg$ is equal to the annihilator of $1\otimes 1\in\Qreg$.

Let $\barEreg := \Ereg\otimes_{\C[\hbar]}\C \cong F\otimes H^2_T(\tX; \C)$ be the algebra obtained from $\Ereg$ by setting
$\hbar = 0$.  Similarly, let $\barQreg := \barEreg\otimes_{\Ereg}\Qreg$ and $\barIreg := \barEreg\otimes_{\Ereg}\Ireg$.
For any $e\in \Ereg$, let $\bar e$ denote its image in $\barEreg$.

\begin{lemma}\label{barIreg}
The ideal $\barIreg\subset\barEreg$ is equal to the annihilator of $1\otimes 1\in\barQreg$.
\end{lemma}

\begin{proof}
Since $\barEreg$ is commutative, we have $$\overline{s(\a)} = (1-q^\a)\overline{[u]^\a} = (1-q^\a)\prod_{i=1}^n u_i^{|\a_i|}$$
for every positive cocircuit $\a$.
Since $(1-q^\a)$ is invertible, this implies that $\barIreg$ is generated by the square-free monomials in $u$ corresponding to
supports of circuits.  This in turn is equal the kernel of the natural map $\Sym H^2_T(\tX; \C)\to H^*_T(\tX; \C)$ \cite[Theorem 2.4]{Konno-eq},
which is by definition the annihilator of $1\otimes 1$.
\end{proof}

\begin{proposition}\label{Ireg}
The ideal $\Ireg\subset\Ereg$ is equal to the annihilator of $1\otimes 1\in\Qreg$.
\end{proposition}

\begin{proof}
The fact that each $s(\a)$ annihilates $1\otimes 1$ is proved in \cite[Proposition 6.4]{McBS}\footnote{The notations of \cite{McBS} compare with ours as follows. The function $q^{S}$ in that paper corresponds to our function $q^{\alpha}$, where $\alpha$ is the root associated to $S$. This differs by a factor of the ``theta characteristic'' $(-1)^{|S|}$ from the function $q^{\beta_S}$ which also appears in \cite{McBS}, but plays no role here. The quantum connection in \cite{McBS} is defined via the formula $\frac{d}{du} + u \star$ rather than $\frac{d}{du} - u \star$, as in this paper. Thus the operator $\mathcal{E}_i$ in \cite{McBS} is conjugate to the operator $u_i$ on $\Qreg$ in this paper via the automorphism of $H_{T\times \mathbb{G}_m}^{*}(\tX ; \C)$ which multiplies an element of $H_{T\times \mathbb{G}_m}^{2n}(\tX ; \C)$ by $(-1)^n$. Finally, the variable $\hbar$ in \cite{McBS} is a primitive character of the dilating torus, whereas for us it is the weight of the symplectic form.}, thus $\Ireg$ is contained
in the annihilator of $1\otimes 1$.  For the opposite inclusion, let $e\in\Ereg$ be a class of degree $k$ that annihilates $1\otimes 1$.
We will prove by induction on $k$ that $e\in\Ireg$.  This is trivial if $k=0$, in which case we must have $e=0$.
For general $k$, we know that $\bar e\in\barEreg$ annihilates $1\otimes 1\in\barQreg$, and therefore Lemma \ref{barIreg}
tells us that $\bar e\in\barIreg$.
This means that there exists some $i\in \Ireg$ of degree $k$
and $e'\in\Ereg$ of degree $k-2$ such that $e = i+\hbar e'$.  Then $\hbar e' = e-i$ annihilates $1\otimes 1$.
Since $\Qreg$ is a free module over $\C[\hbar]$, this implies that $e'$ annihilates $1\otimes 1$.
By our inductive hypothesis, $e'\in\Ireg$, therefore $e\in\Ireg$.
\end{proof}

\subsection{Duality}
Recall from Section \ref{sec:basics} that we began with the data of $\gamma\in(\ft_\Z)^n$ satisfying three conditions.
This can be interpreted as a surjective map $\Z^n\to\ft_\Z$, and we thus obtain an exact sequence
$$0 \to \fk_{\Z} \to \Z^n \to \ft_{\Z} \to 0.$$
By dualizing this sequence, we obtain an element $\gamma^!\in(\fk_\Z^*)^n$, satisfying the same three conditions,
known as the {\bf Gale dual} of $\gamma$.  We then have $T^! \cong K^*$ and $K^! \cong T^*$.
Let $X^!$ be the corresponding hypertoric variety, and choose a generic element
$\theta^!\in\ft_\Z$ to obtain a symplectic resolution $\tX^!\to X^!$.
We have
$$(\ft^!)^* \cong \fk \cong H_2(\tX; \C) \and H^2(\tX^!; \C) \cong (\fk^!)^* \cong \ft$$ via the Kirwan maps for $\tX$ and $\tX^!$,
and
$$H^2_{T^!\times\cs}(\tX^!; \C) = \C\{u^!_1,\ldots,u^!_n,\hbar\} \cong \C\{a_1,\ldots,a_n,\hbar\} = \Azo,$$
where the isomorphism takes $u_i^!$ to $a_i$ and $\hbar$ to $\hbar$. These isomorphisms are easily seen to be compatible with the maps in the equivariant and quantum exact sequences, thus the first
item in Section \ref{sec:duality} is satisfied.  The isomorphism $\ft_\Z^*\cong\fk_\Z^!$ induces a bijection between cocircuits for $\gamma$
and circuits for $\gamma^!$, and therefore between equivariant roots for $X$ and K\"ahler roots for $\tX^!$.
If we choose $\theta^! = \xi$, then the positive equivariant roots match the positive K\"ahler roots, and the second
item in Section \ref{sec:duality} is satisfied.  A more formal proof of symplectic duality between $X$ and $X^!$ 
appears in \cite{BLPWgco}[Theorem 10.8].

\begin{theorem}\label{hypertoric-pi}
Conjecture \ref{pi} holds for hypertoric varieties.
\end{theorem}

\begin{proof}
Proposition \ref{JJ'} tells us that $M \cong R\; /\; R\cdot \{ r(\la) \mid \la\in\Sigma_+ \}$, and therefore that
$\Mreg \cong \Rreg\; /\; \Rreg\cdot \{ r(\la) \mid \la\in\Sigma_+ \}$.
On the other hand, Proposition \ref{Ireg} tells us that $\Qreg^! \cong \Sreg^!\; /\; \Sreg^!\cdot \{ s(\a) \mid \a\in\Delta_+ \}$.
We know that $\Sreg^!\cong\Rreg$ and that the isomorphism $\ft_\Z^*\cong\fk_\Z^!$ takes $\Sigma_+$ to $\Delta_+$,
thus the theorem follows from the fact that $r(\la)$ and $s(\a)$ are defined by the same formula. 
\end{proof}

\section{The Springer resolution}
Our goal in this section to prove Conjecture \ref{piW} for the Springer resolution.

\subsection{The algebraic Harish-Chandra map}
Let $ G $ be a semisimple complex group, which we assume to be of adjoint type.
Following the notation of Example \ref{Springer example}, we let $ X := \mathcal N$ and 
$$ \scrA := \Uhg \otimes_{Z(\Uhg)} \big(\Sym \ft \otimes \C[\hbar]\big),$$ 
which implies that $$\Azo = \ft \oplus \C\hbar \oplus \ft. $$
In particular, the quantization exact sequence \eqref{eq:quant-exact} naturally splits.
The Namikawa Weyl group coincides with the usual Weyl group $ W $ of $ G $, and we have
$$ \scrA^W \cong \Uhg\and (\Sym \Azo)^W \cong \Sym \ft \otimes \C[\hbar] \otimes (\Sym \ft)^W. $$

We define the {\bf Weyl vector} $\rho\in \frac{1}{2}\ft^*_\Z$ to be half the sum of the positive roots,
and the {\bf dual Weyl vector} $\xi\in\ft_\Z$ to be half the sum of the positive coroots.
The equivariant roots $\Sigma\subset\ft^*$ of $X$ coincide with the roots in the usual Lie theory sense,
and the positive equivariant roots $\Sigma_+\subset\Sigma$ (those that pair positively with $\xi$) coincide with the positive roots
in the Lie theory sense.

For every element $ \lambda \in \ft^* $, we have an evaluation map $ \Sym \ft \rightarrow \C $.  We can apply the Rees construction to this map and obtain a graded $\C[\hbar]$-algebra map $ \Sym \ft \otimes \C[\hbar] \rightarrow \C[\hbar] $ which we denote by $ y \mapsto y(\lambda) $.
For any $ \fg$-module $V$, we obtain a module $ V_\hbar $ over $ \Uhg $ by taking the Rees construction (with respect to the trivial filtration on $ V $).  If $ V $ is indecomposable (for example a simple module or a Verma module), then 
every $ a \in Z(\Uhg) $ acts on $V_\hbar$ by some scalar 
$ a(V) \in \C[\hbar] $.  The resulting map $ a \mapsto a(V) $ is a graded $\C[\hbar]$-algebra map $ Z(\Uhg) \rightarrow \C[\hbar] $.
Let $ y \mapsto y_\rho $ 
be the graded $ \C[\hbar]$-algebra automorphism of $ \Sym \ft \otimes \C[\hbar] $ defined by putting $ x_\rho = x - \langle \rho, x \rangle \hbar $
for all $x\in \ft$.  In particular, for any $ \lambda \in \ft^* $ and $y\in \Sym \ft \otimes \C[\hbar] $, we have $ y_\rho(\lambda) = y(\lambda - \rho)$.

We will refer to a finite-dimensional irreducible $\fg$-module as a {\bf \boldmath{$G$}-irrep}.  
Such modules are classified by dominant weights; for any dominant weight $ \lambda\in\ft^*_\Z$, 
let $V(\lambda) $ be the $G$-irrep of highest weight $ \lambda $.
The {\bf algebraic Harish-Chandra map} is the unique graded $ \C[\hbar]$-algebra map $$\xiHC : Z(\Uhg) \rightarrow \Sym \ft \otimes \C[\hbar]$$ 
with the property that, for any dominant weight $ \lambda $ and any $ a \in Z(\Uhg)  $, $ \xiHC(a)(\lambda+ \rho) = a(V(\lambda))$.

Recall from Section \ref{sec:specializations} that we have
$$B(\Uhg) := (\Uhg)_0\Big{/} \sum_{\langle\mu,\xi\rangle > 0}
\left\{a b \mid a \in (\Uhg)_\mu, b \in (\Uhg)_{-\mu}\right\}.$$
We have a natural map $\psi: \Sym \ft \otimes \C[\hbar] \rightarrow B(\Uhg) $ coming from the inclusion $ \ft \rightarrow \fg $.  This map and the algebraic Harish-Chandra map are close to being mutually inverse isomorphisms.  More precisely, we have the following standard results, see for example \cite[Section 23.3]{Humphreys}.
\begin{proposition} \label{BSymft}
The maps $\varphi$ and $\psi$ have the following properties.
\begin{enumerate}
\item The map $ \xiHC $ is injective with image $ (\Sym \ft)^W \otimes \C[\hbar] $.
\item The map $\psi $ is an isomorphism.  
\item For any element $ a  \in Z(\Uhg) \subset B(\Uhg)$, we have $ \psi(\xiHC(a)_\rho) = a$.
\end{enumerate}
\end{proposition}

By Proposition \ref{BSymft}(1), we may use the algebraic Harish-Chandra map 
$ \xiHC $ to endow $ \Uhg $, $ B(\Uhg)$, and other related objects with an action of $ (\Sym \ft)^W \otimes \C[\hbar]$.

\begin{example}
Consider $ \fg = \sl_2 $.  Then $ Z(\Uhg) = \C[C, \hbar] $ where 
$$ C := 2EF + 2FE + H^2 + \hbar^2 = 4FE + H^2 + 2\hbar H + \hbar^2 = 4EF + H^2 - 2\hbar H + \hbar^2.$$  
We have $\Sym \ft = \C[H] $ and $ W = S_2 $ acts by negating $H$, so $ (\Sym \ft)^W \otimes \C[\hbar] = \C[H^2, \hbar] $.
Identify $\ft^*_\Z$ with $\Z$ by sending $\rho$ to 1.  Then we have
$$\hbar^2(n+1)^2 = C(V(n)) = \xiHC(C)(n+1),$$
which implies that $ \xiHC(C) = H^2$.
In $ B(U_{\hspace{-1pt}\hbar}\hspace{1pt} \sl_2) $, the element $4EF$ goes to 0, so the image of $ C $ in $ B(U_{\hspace{-1pt}\hbar}\hspace{1pt} \sl_2) $ is $ (H-\hbar)^2 $.
\end{example}

\subsection{Equivariant Hikita}
Let $G^!$ be the Langlands dual of $G$. 
  Let $$X^! := \mathcal{N}^!\and \tX^! := T^*(G^!/B^!).$$
  Since $ G $ was assumed to be of adjoint type, we see that $G^!$ is simply-connected, so $ G^! $ will typically contain a finite centre and thus not be equal to the automorphism of $ X^!$.  However, this will not affect the equivariant cohomology (over $ \C $), so we will work with $ G^! $ 
  rather than $\Aut(X^!)$ in order to simplify our notation.
  
   For any $\la \in \ft_\Z = (\ft^!)_\Z^*$, let $L(\la) := G^! \times \C_{\lambda} / B^!$ be the associated line bundle on $G^!/B^!$. This line bundle carries a unique $G^!$-equivariant structure and therefore also a canonical $T^!$-equivariant structure.
We endow it with a $G^!\times\cs$-equivariant structure by letting $\cs$ act with weight $ \langle 2\rho,\la\rangle \in \Z$,
and we let $\tilde L(\la)$ denote the pullback of $L(\la)$ to $\tX^!$.  (This non-standard choice of $\cs$-equivariant structure on the line bundle will be necessary later for the equivariant Hikita isomorphism.)

We obtain an isomorphism $ \ft_\Z = (\ft^!)_\Z^* \cong H^2(\tX^!; \Z)$ by sending $\lambda$ to the (non-equivariant) Chern class of  $\tilde L(\la)$.  Moreover, the map taking $\la$ to the $T^!\times\cs$-equivariant first Chern class of $\tilde L(\la)$ provides a splitting of the 
cohomology exact sequence \eqref{eq:coh-exact}.  In particular, we have a canonical isomorphism
$$
H^2_{T^! \times \cs}(\tX^!; \C) \cong  H^2(\tX^!; \C)  \oplus H^2_{T^!\times\cs}(pt; \C) \cong \ft \oplus \C\hbar \oplus \ft.
$$
Taking symmetric algebras and $W$-invariants, we also have
$$
\Big(\Sym H^2_{T^! \times \cs}(\tX^!; \C)\Big)^W \cong \Sym \ft \otimes \C[\hbar] \otimes (\Sym \ft)^W.$$

\begin{proposition} \label{HSymft}
Consider the graded $\C[\hbar]$-algebra homomorphism $$\cohpsi : \Sym\ft\otimes\C[\hbar]\to H^*_{G^!\times\cs}(\tX^!; \C)$$
taking $\la \ft_\Z$ to the $G^!\times\cs$-equivariant first Chern class of $\tilde L(\la)$.
\begin{enumerate}
\item This map is an isomorphism.
\item
The inclusion 
$$(\Sym \ft)^W \otimes \C[\hbar] \cong H^*_{G^! \times \cs}(pt; \C) \rightarrow H^*_{G^! \times \cs}(\tX^!; \C) \cong \Sym \ft \otimes \C[\hbar]$$
takes $  y \in (\Sym \ft)^W \otimes \C[\hbar] $ to $ y_\rho \in \Sym \ft \otimes \C[\hbar]$.
\end{enumerate}
\end{proposition}

\begin{proof}
The first statement follows from the fact that the projection from $T^*(G^!/B^!)$ to $G^!/B^!$ is a homotopy equivalence,
and we have 
$$H^*_{G^!\times\cs}(G^!/B^!; \C) \cong H^*_{G^!\times B^!\times\cs}(G^!; \C) \cong H^*_{B^!\times\cs}(pt; \C) \cong \Sym\ft\otimes\C[\hbar].$$ 
The second statement follows from the way in which we defined the action of $\cs$ on $\tilde L(\la)$.\end{proof}

We now check that the $W$-invariant version of the equivariant Hikita conjecture (Remark \ref{Nak-Hik}) holds for the Springer resolution.  

\begin{proposition}
There is an isomorphism $B(\Uhg)\cong H_{G^! \times \cs}^*(\tX^!)$ of graded algebras over the ring 
$ \Sym \ft \otimes \C[\hbar] \otimes (\Sym \ft)^W $.
\end{proposition}

\begin{proof}
Propositions \ref{BSymft}(2) and \ref{HSymft}(1) tell us
that both rings are isomorphic to $\Sym\ft\otimes\C[\hbar]$.
The action of $\Sym\ft$ on both rings is the obvious one, but the action of $\C[\hbar] \otimes (\Sym \ft)^W$ is not so obvious.
Propositions \ref{BSymft}(3) and \ref{HSymft}(2) tell us that, for both algebras,
an element $ y \in  (\Sym \ft)^W \otimes \C[\hbar] $ is mapped to $ y_\rho \in \Sym \ft \otimes \C[\hbar]$.
\end{proof}

\subsection{Differential operators}\label{sec:do}
By Langlands duality we have $\ft^*_\Z = \ft^!_\Z \cong H_2(\tX^!; \Z)$. By \cite[Theorem 1.1]{BMO}, under this identification, the positive K\"ahler roots $\Delta^!_+\subset  H_2(\tX^!; \Z)$ coincide with 
the positive roots of $G$.  In particular, this means that $\Delta^!_+ = \Sigma_+$
and $$\Freg^! = \C[q^\a, (1- q^{\alpha})^{-1} \mid \alpha \in \Sigma_+] = \Sreg.$$

Let $D_\hbar(\Treg)$ be the Rees algebra of the ring of differential operators on $\Treg := T\cap \Spec \Sreg$, filtered by order.
Applying the Rees construction to the action of differential operators on functions, we obtain an action of
$ D_\hbar(\Treg) $ on $ \cO(\Treg)\otimes\C[\hbar] $.
Each element $ x \in \ft $ gives rise to an invariant vector field on $ \Treg $, which induces a map 
$ \iota : \Sym \ft \otimes \C[\hbar] \rightarrow D_\hbar(\Treg)$, and we have
\begin{equation}\label{iota}\iota(y)\cdot q^\lambda = y(\lambda) q^\lambda\end{equation}
for all $ y \in \Sym \ft \otimes \C[\hbar]$ and $\lambda \in \ft^*_\Z $.

Let $ \Dreg $ be the $ \C[\hbar]$-subalgebra of $ D_\hbar(\Treg) $ generated by the images of $ \iota $ and $ \Sreg \subset \cO(\Treg) $. 
We then have a graded vector space isomorphism $\Dreg \cong \Sreg \otimes \Sym \ft \otimes \C[\hbar] $.
We can also regard $ \Dreg $ as a subalgebra of $ \Rreg^W \cong (\Ereg^!)^W$, and we have graded algebra isomorphisms
$$
\Rreg^W \cong \Dreg \otimes (\Sym \ft)^W \cong (\Ereg^!)^W.
$$

\subsection{The geometric Harish-Chandra map}

We thank Sam Gunningham for help with the proof of the following lemma.

\begin{lemma}\label{Gun}
Let $ X$ and $Y $ be smooth affine varieties over $ \C $ and let $X \rightarrow Y $ be a dominant morphism of relative dimension zero.  
Suppose that $ d \in D(X) $ and $ d\cdot f = 0 $ for all $ f \in \mathcal{O}(Y) \subset\cO(X) $.  Then $ d = 0 $.
\end{lemma}

\begin{proof}
We will prove the lemma by induction on the order of $ d $.  When the order is zero, $ d $ is multiplication by a function, so the result holds.
Now suppose that the lemma holds for differential operators of order at most $k-1$, and
let $ d $ be a differential operator of order at most $ k $ such that $d\cdot f = 0$ for all $f\in\cO(Y)$.  For any $ f \in \mathcal{O}(Y)$, the commutator $ [d,f] \in D(X) $ is a differential operator of order at most $ k-1$.  For any $ g \in \mathcal{O}(Y) $, we have $$ [d,f](g) = d\cdot (fg) - fd\cdot g = 0. $$
Thus our inductive hypothesis tells us that $ [d,f] = 0 $.

The ring $D(X)$ acts faithfully on the function field $\mathcal{K}(X)$, 
and the above paragraph implies that the element $d\in D(X)$ acts $\mathcal{K}(Y)$-linearly,
and thus can be regarded as an element of $ D(\mathcal{K}(X)/\mathcal{K}(Y))$.  By the smoothness assumption, $ D(\mathcal{K}(X)/\mathcal{K}(Y))$ is generated by $ \mathcal{K}(X) $ and $\mathcal{K}(Y)$-linear derivations of $ \mathcal{K}(X)$.  
Since our map has relative dimension zero,
$ \mathcal{K}(Y) /\mathcal{K}(X)$ is an algebraic extension, which implies that there are no such derivations.  Thus $ d \in \mathcal{K}(X) $ and the result follows.
\end{proof}

\begin{corollary} \label{diszero}
 If $ d \in D_\hbar(\Treg) $ and $ d\cdot \chi_V = 0 $ for all $G$-irreps $V$, then $ d = 0 $.
\end{corollary}

\begin{proof}
 We apply the Lemma \ref{Gun} to the map $ \Treg \rightarrow T \rightarrow T/W $.  Since the characters of irreducible representations form a basis 
 for $ \cO(T/W)$, the result follows.
\end{proof}

By Corollary \ref{diszero}, there is a unique graded $\C[\hbar]$-algebra homomorphism $$ \GHC : Z(\Uhg) \rightarrow D_\hbar(\Treg) $$
with the property that, for any $G$-irrep $ V $ and any $ a \in Z(\Uhg)  $, we have $ \GHC(a)(\chi_V) = a(V) \chi_V $.
We call this homomorphism the {\bf geometric Harish-Chandra map}.
Let $$\delta := \prod_{\alpha \in \Delta_+} \left(q^{\alpha} -1\right)$$ denote the Weyl denominator.
The algebraic Harish-Chandra map $\xiHC $ and the geometric Harish-Chandra map $\GHC $ are related by the following lemma.  

\begin{lemma} \label{remHC}
 For any $ a \in Z(\Uhg) $, we have $$
 \GHC(a) = \delta^{-1} \iota(\xiHC(a)_\rho) \delta.$$
 In particular, the image of $ \GHC $ is contained in the subalgebra $ \Dreg \subset D_\hbar(\Treg) $.

\end{lemma}

\begin{proof}
We need to show that
$$\delta^{-1} \iota(\varphi(a)_\rho)\delta\cdot \chi_{V(\la)} = a(V(\la))\cdot \chi_{V(\la)}$$
for every dominant weight $\la$.  The Weyl character formula says that
$$
\chi_{V(\lambda)} = \sum_{w \in W} (-1)^{\ell(w)}\frac{q^{w(\lambda + \rho) + \rho}}{\delta},
$$
and therefore we need to show that
$$\iota(\varphi(a)_\rho)\cdot \sum_{w \in W} (-1)^{\ell(w)}q^{w(\lambda + \rho) + \rho} 
= a(V(\la)) \cdot \sum_{w \in W} (-1)^{\ell(w)}q^{w(\lambda + \rho) + \rho}.$$
We will prove the equality summand-by-summand. 
Equation \eqref{iota} tells us that
$$\iota(\varphi(a)_\rho)\cdot q^{w(\lambda + \rho) + \rho}
= \varphi(a)_\rho(w(\lambda + \rho) + \rho) \cdot q^{w(\lambda + \rho) + \rho}
= \varphi(a)(w(\lambda + \rho)) \cdot q^{w(\lambda + \rho) + \rho}.$$
By Weyl invariance of $ \xiHC(a) $, this is equal to
$\varphi(a)(\lambda + \rho) \cdot q^{w(\lambda + \rho) + \rho},$
which by the definition of $\xiHC$ is equal to $a(V(\la)) \cdot q^{w(\lambda + \rho) + \rho}.$
This concludes the proof.
\end{proof}

\subsection{The D-module of traces}
This section is devoted to computing the $\Rreg^W$-module
$$
\Mreg^W = \Sreg \otimes (\Uhg)_0 \Big{/} \sum_{\la\in\N\Sigma_+}\Sreg[\hbar]\cdot\left\{ 1\otimes a b - q^\lambda \otimes b a \mid a \in (\Uhg)_\lambda, b \in (\Uhg)_{-\lambda}\right\}.
$$
We begin by proving that, as a module over the subalgebra $\Dreg\subset\Rreg^W$, $\Mreg^W$ is isomorphic to the regular module.

\begin{theorem} \label{ComputationMreg}
The map $ \sigma : \Dreg \rightarrow \Mreg^W$ taking $ d\in\Dreg$ to $d \cdot (1 \otimes 1) \in \Mreg^W$ is an isomorphism of 
graded $ \Dreg $-modules.
\end{theorem}

\begin{proof}
We begin by showing that $\sigma$ is surjective, which we will prove by induction on degree. 
Assume that $\sigma$ is surjective in all degrees less than $k$, and let $a\in(\Uhg)_0^k$. 
Write $ a = h + \sum_{\alpha \in \Delta_+} E_\alpha b_\alpha $, where $ h \in \Sym \ft $ has degree $k$ and $ b_\alpha \in (\Uhg)_{-\alpha}^{k-2}$.  It is clear that the image of $h$ in $\Mreg^W$ lies in the image of $\sigma$,
so it is enough to prove that the images of each $E_\alpha b_\alpha$ in $\Mreg^W$ lie in the image of $\sigma$, as well.

We know that there exists an element $c\in (\Uhg)_0^{k-2}$ such that $[b_\alpha, E_\alpha ] = \hbar c$.
In $\Mreg^W$, we have 
$$ E_\alpha b_\alpha = q^\alpha b_\alpha E_\alpha = q^\alpha \left(E_\alpha b_\alpha + [b_\a, E_\a]\right) 
= q^\alpha \left(E_\alpha b_\alpha + \hbar c\right),$$
and therefore
$$ E_\alpha b_\alpha = \frac{q^\alpha}{1-q^\alpha} \hbar c.
$$
By our inductive hypothesis, $ c $ lies in the image of $\sigma $, and thus so does  $E_\alpha b_\alpha$.  This proves surjectivity.

To show injectivity of $ \sigma $, recall from Proposition \ref{lemma characters} that, for any $G$-irrep $ V $, 
the map $ \tr_V : \Mreg^W \rightarrow \cO(\Treg)\otimes \C[\hbar] $ is a $ \Dreg $-module map.  
Suppose that $ d \in \Dreg $ and that $ \sigma(d) = 0 $.  Then for any $G$-irrep $ V $, we have
$$
0 =  \tr_V(0) = \tr_V(\sigma(d)) = \tr_V(d \cdot (1 \otimes 1)) = d\cdot \tr_V(1\otimes 1) = d\cdot \chi_V.
$$
Then Lemma \ref{diszero} tells us that $ d = 0 $.
 \end{proof}

It remains only to determine how $(\Sym \ft)^W \otimes \C[\hbar]\subset \Rreg^W$ acts on $\Mreg^W$.

\begin{lemma}\label{first}
For all $a\in Z(\Uhg)$ and $d\in \Dreg\cong \Mreg^W$, we have $a\cdot d = \GHC(a) d$.
\end{lemma}

\begin{proof}
Let $\GHC':Z(\Uhg)\to\Dreg$ denote the composition
$$
Z(\Uhg) \rightarrow (\Uhg)_0 \rightarrow \Mreg^W \xrightarrow{\sigma^{-1}} \Dreg.
$$
We wish to show that $ \GHC' = \GHC $.
Fix an element $ a \in Z(\Uhg)$.  For any $G$-irrep $V$, Proposition \ref{lemma characters} implies that
$$\GHC'(a)\cdot \chi_V = \GHC'(a) \cdot \tr_V(1) = \tr_V(a) = a(V) \chi_V = \GHC(a)\cdot \chi_V,$$
thus  $ (\GHC(a) - \GHC'(a))\cdot \chi_V = 0 $.  By Lemma \ref{diszero}, we conclude that $ \GHC'(a) = \GHC(a) $.
\end{proof}

Lemmas \ref{remHC} and \ref{first} combine to give us the following result.

\begin{proposition} \label{co:CentreAction}
For all $y\in (\Sym \ft)^W \otimes \C[\hbar]$ and $d\in \Dreg\cong \Mreg^W$,
we have $$y\cdot d = \delta^{-1} \iota(y_\rho) \delta d.$$
\end{proposition}

\subsection{The quantum D-module}

Consider the graded $\Dreg$-module homomorphism $\Psi:\Dreg \to (\Qreg^!)^W$ taking $ d\in\Dreg$ to $d \cdot (1 \otimes 1)$.

\begin{theorem} \label{ComputationQreg}
The map $\Psi$ is an isomorphism of $ \Dreg $-modules.  Moreover, 
for any $a\in (\Sym \ft)^W\cong H^*_{G^!}(pt; \C)$, the image of $a$ in $(\Qreg^!)^W$ under the natural inclusion
of the equivariant cohomology of a point into the equivariant cohomology of $T^* (G^!/B^!)$
is equal to $\Psi(\delta^{-1}\iota(a_\rho) \delta)$.
\end{theorem}

Before proving this result, we recall the results of Braverman-Maulik-Okounkov $ \cite{BMO}$ describing the quantum connection of $T^*(G^!/B^!)$. We then relate this connection to the AKZ connection studied by Cherednick and Matsuo, and use this relationship to prove certain differential relations are satisfied in $(\Qreg^!)^W$. With these relations in hand, we will conclude by proving Theorem \ref{ComputationQreg}.

In order to cite the results of  $\cite{BMO}$, it will be convenient to use the more standard shifted isomorphism $\Sym \ft \otimes \C[\hbar] \cong H^*_{G^! \times \cs}\big(T^* (G^!/B^!); \C\big)$ defined by $\la \mapsto \cohpsi(\la) - \hbar \langle \rho, \la \rangle$ (for $ \lambda \in \ft_\Z$). This isomorphism induces an isomorphism of graded vector spaces
\begin{equation} \label{eq:qregpresentation}
(\Qreg^!)^W := \Sreg \otimes H^*_{G^! \times \cs}\big(T^* (G^!/B^!); \C\big) \cong \Sreg \otimes \Sym \ft \otimes \C[\hbar].
\end{equation}
We will now describe the induced $(\Ereg^!)^W$-module structure on the right-hand space.  Recall from Section \ref{sec:do} that $ (\Ereg^!)^W \cong \Dreg \otimes (\Sym \ft)^W$, and that $ \Dreg $ is generated 
over $\C[\hbar]$ by $\Sreg $ and $ \ft $. 
$ \Sreg \subset \Dreg$ acts on $S_\reg \otimes \Sym \ft \otimes \C[\hbar]$ by multiplication on the left factor, and $ (\Sym \ft)^W $ acts by multiplication on the right factor. Following \cite{BMO}, we will describe the action of $x \in \ft$ in terms of the {\bf degenerate Hecke algebra} $\mathcal{H}_{\hbar}$. This is the algebra generated by $\C[W]$ and $\Sym \ft \otimes \C[\hbar]$, subject to the relations
\[ s_{\alpha} x - s_{\alpha}(x) s_{\alpha} = -\hbar \langle \alpha, x \rangle, \] where $s_{\alpha}$ is the reflection associated with the simple root $\alpha$ and $x \in \ft$. Here our $\hbar$ corresponds to the variable $-t$ in \cite{BMO}.

There is a natural identification of $\Sym \ft \otimes \C[\hbar]$ with the $ \mathcal H_\hbar $-module 
$ J := \mathcal{H}_\hbar \otimes_{\C[W]} \C$, where $\C[W]$ acts on $\C$ via the trivial representation, inducing an identification  of the right-hand side of Equation \ref{eq:qregpresentation} with $S_\reg \otimes J$.
By our definition of $\Qreg^!$, the element $x \in \ft$ acts via the covariant derivative $\hbar \partial_x \qconnsign \cohpsi(x) \star$. By \cite[Theorem 3.2]{BMO}, the latter acts on $(\Qreg^!)^W \cong S_\reg \otimes J$  via the operator
\begin{equation} \label{BMODmodule}
 \nabla^{BMO}_{x} := \hbar \partial_x \qconnsign (x + \hbar \langle \rho, x \rangle) \qconnosign \hbar \sum_{\alpha \in \Delta_+} \langle \alpha, x \rangle \frac{q^{\alpha}}{1-q^\alpha}(s_\alpha-1),
\end{equation}
where $\partial_x$ acts on the left factor and $x$ and $s_{\alpha}$ are viewed as elements of $\mathcal{H}_{\hbar}$, which acts on the right factor.  

The $\Sreg$-module $(\Qreg^!)^W$ has infinite rank. In order to apply certain results on holonomic D-modules, we will need to base-change to $T_\reg \subset \Spec \Sreg$ and consider various finite rank quotients, obtained by specialising equivariant parameters, or equivalently central characters of the Hecke algebra. The center of $\mathcal{H}_{\hbar}$ is equal to $(\Sym \ft)^W \otimes \C[\hbar]$.  Thus, given any $c \in \ft^* / W$ and $\cmt \in \C$, we can consider the corresponding evaluation module $\C_c$ of $(\Sym\ft)^W$, and define the $\mathcal{H}_{\hbar}$-module
$$J^{c,\cmt} := \Sym \ft \otimes_{(\Sym \ft)^W} \C_c \otimes_{\C[\hbar]} \C[\hbar]/(\hbar+\cmt)$$
and the $D_{\hbar}(\Treg)$-module $$Q^{c,t} := (Q_{T_\reg}^{!})^W \otimes_{(\Sym \ft)^W} \C_c \otimes_{\C[\hbar]} \C[\hbar]/(\hbar+\cmt).$$ As a vector space, we have an isomorphism $Q^{c,\cmt} \cong \mathcal{O}(T_\reg) \otimes J^{c,\cmt}$. Note that $J^{c,\cmt}$ has complex dimension $|W|$.

\begin{definition} \cite[Definition 1.1.39]{cherednik2005double}
The affine Knizhnik-Zamolodchikov (AKZ) connection is the following $\mathcal{H}_\hbar$-valued connection on $\Treg$ : 	
\begin{equation} \label{AKZconnection} \nabla_x^{\AKZsup} := \partial_x \qconnsign x \qconnosign \hbar \sum_{\alpha \in \Delta_+} \langle \alpha, x \rangle \frac{s_\alpha}{q^\alpha-1}. \end{equation}
  Let $\mathcal{M}_{\AKZ}  := \cO(\Treg) \otimes J$ be the $\Dd(\Treg)$-module on which $\iota(x)$ acts via $\nabla_x^{\AKZsup}$. As above, we write $\mathcal{M}^{c,\cmt}$ for the corresponding specialisations at $(c,\cmt) \in \ft^* / W \times \C$.
\end{definition}

Let $\pi$ be the $\C$-linear involution of $\cO(\Treg) \otimes J$ defined by $$\pi( f(q) \otimes a) = f(q^{-1}) \otimes a. $$ Note that $\pi \circ \partial_x \circ \pi = - \partial_x$.
 
\begin{lemma} \label{conjugationlemma}
	The isomorphism $Q^{c,1} \cong \mathcal{O}(T_\reg) \otimes J^{c,1}$ identifies the action of $\delta^{-1} \nabla_{x_\rho}^{BMO} \delta$ with the action of $\pi \circ \nabla^{\AKZsup}_x \circ \pi$ specialised at $\hbar=-1$. 
\end{lemma}
\begin{proof}
	In the ring $\Ereg^!$, we have the equality \begin{align*} \delta^{-1} \tilde \psi(x_\rho) \delta  = \tilde \psi(x_\rho) + \hbar \sum_{\alpha \in \Delta^+} \langle \alpha, x \rangle \frac{q^{\alpha}}{q^{\alpha}-1}. \end{align*}
The action of this element on $(\Qreg^!)^{W}$ is given, via \ref{BMODmodule}, by
	\begin{align*} &\phantom{= } \hbar\partial_x \qconnsign x \qconnosign \hbar \sum_{\alpha \in \Delta_+} \langle \alpha, x \rangle \frac{q^{\alpha}}{1 - q^{\alpha}}(s_\alpha-1) \qconnosign  \hbar \sum_{\alpha \in \Delta^+} \langle \alpha, x \rangle \frac{q^{\alpha}}{1-q^{\alpha}}	\\ 
	&=  \hbar \partial_x \qconnsign x \qconnosign \hbar \sum_{\alpha \in \Delta_+} \langle \alpha, x \rangle \frac{q^{\alpha}}{1 - q^{\alpha}} s_\alpha \\
	 &=  \hbar \partial_x \qconnsign x \qconnosign  \hbar\sum_{\alpha \in \Delta_+} \langle \alpha, x \rangle \frac{s_\alpha}{q^{-\alpha}-1}. \end{align*}
	Conjugating by $\pi$ gives
	\[  \qconnsign \hbar \partial_x \qconnsign x \qconnosign \hbar\sum_{\alpha \in \Delta_+} \langle \alpha, x \rangle \frac{s_\alpha}{q^{\alpha}-1}\]and setting $\hbar=-1$ yields
	\[  \partial_x \qconnsign x  \qconnsign  \sum_{\alpha \in \Delta_+} \langle \alpha, x \rangle \frac{s_\alpha}{q^{\alpha}-1}.  \] 
\end{proof}
Via Lemma \ref{conjugationlemma}, we will derive the differential relations of Theorem \ref{ComputationQreg} from known relations in $\mathcal{M}^{c, \cmt}_{\AKZ}$. 

The key tool is the identification, due to Cherednik and Matsuo, of $\mathcal{M}_{\AKZ}^{c,\cmt}$ with a certain scalar $D(\Treg)$-module which Cherednik calls the quantum many body problem, and which \cite{BMO} refer to as the trigonometric Calogero-Moser system (see \cite{opdam1998lectures} for an introduction). It may be described as follows. Fix a fixed Weyl invariant non-degenerate quadratic form on $\ft$. The {\bf Calogero-Moser module} $\mathcal{CM}^\cmt$ is a module over the ring $\Dd(\Treg) \otimes (\Sym \ft)^W $. As a module over $\Dd(\Treg)$, it is simply $\Dd(\Treg)$ itself. The action of $a \in (\Sym \ft)^W$ is given by right-multiplication by $\eta_\cmt(a) \in \Dreg$, where $\eta_{\cmt} : (\Sym \ft)^W \to  \Dreg$ is the unique algebra homomorphism satisfying the following two properties.
\begin{enumerate}
	\item $\eta_{\cmt}(a)$ lies in the algebra generated by $\iota(b), b \in \Sym \ft$ and $\frac{1}{1 - q^{\alpha}}$, $\alpha \in \Delta_+$, and has an expansion
	\[ \eta_{\cmt}(a) = \iota(a) + \sum_{\alpha \in \Z^{>0} \Delta_+} q^{\alpha} p_{\alpha}(a) \]
	where the $p_{\alpha}(a)$ are constant coefficient differential operators.
	\item Let $C$ be the element of $(\Sym^2 \ft)^W$ determined by our fixed quadratic form. Let $L \in \Dd(\Treg)$ be the associated constant coefficient differential operator, sometimes called the Laplacian.
\[ \eta_{\cmt}(C) =  \iota(C) - \cmt(\cmt-1)\sum_{\alpha \in \Delta_+} \frac{(\alpha, \alpha )}{(q^{\alpha/2} - q^{-\alpha/2})^2}.\]
\end{enumerate}
The relation $[\eta_{\cmt}(a), \eta_{\cmt}(C)] = 0$ imposes a recursive relation on the coefficients $p_{\alpha}(a)$, which fixes them uniquely. 
\begin{lemma}
$\eta_1(a) = \iota(a)$.
\end{lemma} 
\begin{proof}
The right-hand side defines a homomorphism which clearly satisfies both conditions above.    \end{proof}
As above, given $c \in \ft / W$, we define $\mathcal{CM}^{c,\cmt} := \mathcal{CM}^t \otimes_{(\Sym \ft)^W} \C_c$. This specialisation imposes the relations 
\[ \eta_t(a) \cdot 1 = a(c) \]
for all $a \in (\Sym \ft)^W$. 

Consider the inclusion $1 \to (\Sym \ft)^W$. This extends to a $\Dd(\Treg)$-linear map $\Psi^t : \Dd(\Treg) \to \mathcal{M}^t_{\AKZ}$, taking $1$ to $1 \otimes 1$. 

\begin{theorem} \label{AKZQMBP} The map $\Psi^{\cmt}$ induces a homomorphism of $\Dd(\Treg)$-modules
\[ \Psi^{c,\cmt} : \mathcal{CM}^{c,\cmt} \to \mathcal{M}_{\AKZ}^{c,\cmt}. \]
\end{theorem}
\begin{proof}
The proposition is essentially \cite[Theorem 1.2.11]{cherednik2005double} (see also \cite{matsuo1992integrable}). Since both references state their results in terms of sheaves of solutions rather than D-modules, we make the translation here. To do so, we now switch from algebraic D-modules to D-modules on the analytic space associated to $\Treg$.   We write $ \mathcal D_{\an} $ for the sheaf of differential operators on this space and $ \mathcal O_{\an} $ for the sheaf of analytic functions.  Similarly, $\mathcal{M}_{\AKZ}^{c,\cmt}$ and $\mathcal{CM}^{c,\cmt}$ have analytic versions denoted $ \mathcal{M}^{c,\cmt}_{\an}$ and  $\mathcal{CM}^{c,\cmt}_{\an} $.

The homomorphism $\Psi^{\cmt}$ induces a map of sheaves $$\Psi^{\cmt, c}_{\an} : \mathscr{H}om_{\mathcal D_{\an}} (\mathcal{M}_{\an}^{c,\cmt} , \mathcal{O}_{\an} ) \to  \mathscr{H}om_{\mathcal D_{\an}}(\mathcal D_{\an} , \mathcal{O}_{\an} ),$$ which is given by the formula
$$
\sigma \mapsto ( d \mapsto \sigma(d \cdot 1))
$$
for $ d $ a section of $\mathcal D_{\an}  $ and $ \sigma $ a section of $  \mathscr{H}om_{\mathcal D_{\an}}(\mathcal{M}_{\an}^{c,\cmt} , \mathcal{O}_{\an} ) $.  By \cite[Theorem 1.2.12]{cherednik2005double}, this map factors through $  \mathscr{H}om_{\mathcal D_{\an}}(\mathcal{CM}^{c,\cmt}_{\an} , \mathcal{O}_{\an} ) $, which implies that the map $ \mathcal D_{\an} \rightarrow \mathcal{M}_{\an}^{c,\cmt} $ factors through the projection $ {\mathcal D_{\an}} \rightarrow \mathcal{CM}^{c,\cmt}_{\an}  $.
Taking global sections, we find that $\Psi^{\cmt, c}_{\an}$ factors through the projection from $D_{\an}(\Treg)$ to $\mathcal{CM}^{c,\cmt}_{\an} $.
Since the algebraic modules sit inside of the analytic ones, this implies the statement that we need.
\end{proof} 
It follows from Theorem \ref{AKZQMBP} that for any $a \in (\Sym \ft)^W$, we have the equality
\[ \eta_t(a) \cdot 1 \otimes 1 = 1 \otimes a \] in $\mathcal{M}$. 
Applying $\pi$ to both sides, we get \[ \pi \circ \eta_t(a) \cdot 1 \otimes 1 = \pi \cdot 1 \otimes a. \]
We can rewrite the left-hand side as $\pi \circ \eta_t(a) \circ \pi \cdot 1 \otimes 1$. Specializing at $(c,1)$ and recalling $\eta_1(a) = \iota(a)$, we obtain the equality
	\begin{equation} \label{AKZ1scalar}  \pi \circ \iota(a) \circ \pi \cdot  1 \otimes 1 =  1 \otimes a(c) \end{equation}
 in $\mathcal{M}^{c,1}$. By Lemma \ref{conjugationlemma}, we can identify $\pi \circ \iota(a) \circ \pi $ with the action of $\delta^{-1} \iota(a_\rho) \delta$ on $Q^{c,1}$. We  thus obtain our key result. 
\begin{corollary}
	For any $a \in (\Sym \ft)^W$, the following holds in $Q^{c,1}$:	
	\begin{equation} \label{HCrelation}
\delta^{-1} \iota(a_\rho) \delta  \cdot 1 \otimes 1 = 1 \otimes a(c).
	\end{equation}
\end{corollary}

\begin{proofComputationQreg}
Consider the induced map $\bar \Psi:\Dreg / \hbar \Dreg \to  (\Qreg^!)^W / \hbar  (\Qreg^!)^W$.
We have $\Dreg / \hbar \Dreg \cong \Sreg\otimes\Sym\ft \cong (\Qreg^!)^W / \hbar  (\Qreg^!)^W$.
Under these identifications, Equation \eqref{BMODmodule} implies that $\bar\Psi$ is the unique $\Sreg$-algebra map taking $x$ to $-x$ for all $x\in\ft$.
In particular, $\bar\Psi$ is an isomorphism.
Since $\Dreg $ and $(\Qreg^!)^W $ are graded modules over $ \C[\hbar] $ with bounded below grading and $(\Qreg^!)^W$ is torsion-free as a module over $\C[\hbar]$, we conclude (by a standard argument) that $\Psi$ is also an isomorphism.

Next, we show that for any $a \in (\Sym \ft)^W$, we have $\delta^{-1} \iota(a_\rho) \delta \cdot (1 \otimes 1) = 1 \otimes a$. It is enough to show that this equality holds after localizing to $ \Treg $ and specializing $ (\Qreg^!)^W$ to a generic point in $\Spec \; (\Sym \ft)^W \otimes \C[\hbar]$. By homogeneity, we may assume $\hbar=-1$, i.e. $t=1$ in the notation of this section. The desired equality is then Equation \ref{HCrelation}.
\end{proofComputationQreg}

\vspace{\baselineskip}

The following result now follows immediately from Theorem \ref{ComputationMreg}, Proposition \ref{co:CentreAction} and Theorem \ref{ComputationQreg}.
\begin{theorem}\label{Springer-pi}
We have an isomorphism of $  \Dreg \otimes (\Sym \ft)^W $-modules $ \Mreg^W \cong (\Qreg^!)^W  $, taking
$1\otimes 1$ to $1\otimes 1$.  In particular, Conjectures \ref{pi} and \ref{piW} hold for the Springer resolution.
\end{theorem}

\begin{proof}
The first statement follows from Theorem \ref{ComputationMreg}, Proposition \ref{co:CentreAction} and Theorem \ref{ComputationQreg}.
This establishes that Conjecture \ref{piW} holds for the Springer resolution, and Conjecture \ref{pi} follows from the discussion
in Section \ref{sec:weyl}.
\end{proof}

\appendix 
\section{A geometric description of highest weights} \label{appendixWebster}
\mbox{}\hspace{2cm} {\bf \large an appendix by
  Ben Webster}\\\\
  \newcommand{\mt}{\mathfrak{t}}
  \renewcommand{\cD}[1]{\mathscr{D}^{#1}}
\newcommand{\pa}{c}
\newcommand{\fix}{x}
\newcommand{\fixb}{y}
\newcommand{\comm}{\mathsf{m}}
\newcommand{\anti}{\mathsf{a}}
As in the main text, let $\tilde{X}\to X$ be a conical symplectic
resolution, and  $T$ an algebraic torus acting Hamiltonianly on
$\tilde{X}$.  As before, the algebra $\mathscr{A}$ is the universal deformation
quantization of $X$; it will be more convenient for us to instead work with a specialization of this, though. Based on the work
of Bezrukavnikov and Kaledin, we have a sheaf $\cD{\pa}$ on $\tilde{X}$ quantizing the structure sheaf for each choice of $\pa\in H^2(\tilde{X};\C)$. The global sections of this sheaf 
are $\mathscr{A}^\pa $, the specialization of $ \scrA $ at $ \pa$.

Recall the quantization exact sequence (\ref{eq:quant-exact}) \begin{equation*}
0\to H_2(\tX; \C)\oplus\C\hbar \to \Azo \to \ft \to 0.
\end{equation*}
As discussed in Remark \ref{rem:qex}, a splitting  of this exact sequence is a quantized comoment map $ \ft \rightarrow \Azo $.  Any two such maps differ by an element of $\Hom(\mt, H_2(\tilde{X})\oplus \C\hbar )$.

By universality, we can choose a $T$-equivariant isomorphism $\anti\colon \mathscr{A}
\cong \mathscr{A}^{op}$ sending $\hbar\mapsto -\hbar$.  For any
quantized comoment map $\comm$, we have that $\anti(\comm)$ is again a
quantum comoment map, and this is an anti-automorphism of affine
spaces.  The fixed points form a torsor over $\Hom(\mt,
H_2(\tilde{X}))$;  all the corresponding comoment maps agree after
taking the quotient at $\pa=0$.  

We will be interested in how this comoment map interacts with the
B-algebra mentioned earlier.  This algebra depends on $\xi \in \mt$
and is given by
 $$B(\scrA) := \scrA_0\Big{/} \sum_{\langle\mu,\xi\rangle > 0}
 \left\{a b \mid a \in \scrA_\mu, b \in \scrA_{-\mu}\right\}.$$

 This algebra naturally acts on the weight spaces of any
 $\scrA$-module which minimize $\langle\mu,\xi\rangle$.  Thus, the induced map $\mt \to
 B(\scrA)$ also controls how $\mt$ acts on the lowest weight spaces of
 the different simple modules in category $\cO$.

We use the construction of standard modules in category $\cO$ given in
\cite[\S 5.1]{BLPWgco}.  This is done by analyzing the structure of
$\cD{\pa}$ near any fixed point $\fix$.  By the Darboux theorem, a
formal neighborhood of $\fix$ in $\tilde{X}$ is $T$-equivariantly
symplectomorphic to the tangent space, and thus the completion of
$\cD{\pa}$ here to a completed Weyl algebra (with parameter $\hbar$).
Passing to $T$- and $\mathbb{G}_m$-locally finite vectors gives a copy of the
Weyl algebra $A^{\pa}_{\fix}$, as shown in \cite[Lem. 5.2]{BLPWgco}. Since
we only pass to $\mathbb{G}_m$-locally finite vectors rather than invariants as
in \cite[Lem. 5.2]{BLPWgco}, we obtain a family over $\C[\hbar]$,
which we identify with polynomials on $T_{\fix}\tilde{X}$ with
the Moyal star product.   
The algebra  $A^{\pa}_{\fix}$ has its own $B$-algebra $B^{\pa}_{\fix}\cong
\C[\hbar]$.  Restriction induces a map $\scrA\to A^{\pa}_{\fix}$ which
in turn induces a map $B(\scrA) \to B^{\pa}_{\fix}$.  The induced map
$\Azo\to \C\hbar \subset B^{\pa}_{\fix}$ is exactly the map
$w^{\pa}_{\fix}$ defined in Section \ref{sec:rank}.  

\begin{definition}
For each fixed point $x$,  let $\comm_{\fix}\colon \mt\to \mathscr{A}^2_0$ be the unique
  comoment map whose image in $B^{\pa}_{\fix}$ is 0 for all $\pa$.  
\end{definition}
This comoment map is easily constructed by considering an arbitrary
moment map, and noting that this induces a linear map $\mt \to
H_2(\tilde{X})\oplus \C\hbar$ by considering the values in $B^{\pa}_{\fix}$
for different $\pa$; thus subtracting this, we arrive at
$\comm_{\fix}$.

Of course, this splitting has a dual splitting $\mathscr{A}^2_0\to
H_2(\tX) \oplus \C \hbar$, which is exactly the map $w_{\fix}^{\pa}$,
thinking of $\pa$ as an element of the dual space of $H_2(\tX) \oplus
\C \hbar$.  The differences $\comm_{\fix}-\comm_{\fixb}$ and
$w_{\fix}^{\pa}-w_{\fixb}^{\pa}$ both define maps $\ft\to H_2(\tX) \oplus
\C \hbar$, so we can compare them.
By the usual properties of dual splittings, we have that:
\begin{lemma}\label{lem:comm-w}
  We have an equality
  $\comm_{\fix}-\comm_{\fixb}=w_{\fixb}^{\pa}-w_{\fix}^{\pa}$.  
\end{lemma}


For each fixed point $\fix$, there is a natural decomposition of
$T_{\fix}\tilde{X}=T^+_{\fix}\tilde{X}\oplus T^-_{\fix}\tilde{X}$ according to whether the weight of
$\xi$ is positive or negative.  Let $\{\chi_i\}$ denote the
$T$-weights of $T^+_{\fix}\tilde{X}$ (with multiplicity); of course, since $\xi$ is symplectic, the
weights on $T^-_{\fix}\tilde{X}$ are just $\{-\chi_i\}$.
We define an element $\chi_{\fix}\in \mt^*$ by
$\chi_{\fix}=\nicefrac{1}{2}\sum\chi_i$.  More canonically, we can think
of this as the weight of $T$ on $\det(T^+_{\fix}\tilde{X})^{\nicefrac 12}$.

We have an exact sequence
$$
0 \rightarrow H_2(\tX) \rightarrow H_2^T(\tX) \rightarrow \mt \rightarrow 0 .
$$
This is the dual to the cohomology exact sequence (\ref{eq:coh-exact}), except that we have removed the $ \cs $-equivariance.

For each $\fix$, we have a
splitting $\iota_{\fix}:\mt \to H_2^T(\tilde{X})$ given by push-forward from the
fixed point.  We define
$\rho_{\fix,\fixb}=\iota_{\fix}-\iota_{\fixb}:\mt\to H_2(\tilde{X})$; as usual, the
difference of two splittings is a map to the kernel of the map which
is split. This map also has a dual $\rho^*_{\fix,\fixb}\colon
H^2(\tilde{X})\to \mt^*$.  

\begin{theorem} \label{thm:differenceofhighestweights}
For all fixed points $\fix,\fixb$, we have 
\begin{equation}\label{difference-comoment-maps}
\comm_{\fix}-\comm_{\fixb} =
\rho_{\fix,\fixb} - (\chi_{\fix}-\chi_{\fixb})\hbar \in \Hom(\mt, H_2(\tX) \oplus \C \hbar). 
\end{equation}
That is, for any noncommutative comoment map $\comm$, and any $ u \in \mt$,  the difference
between the action of $\comm(u)$ on $B^{\pa}_{\fixb}$ and on $B^{\pa}_{\fix}$ is
\begin{equation}
(w^{\pa}_{\fixb}-w^{\pa}_{\fix})(u)=\langle u,\rho^*_{\fix,\fixb}(\pa) -(\chi_{\fix}-\chi_{\fixb})\rangle\hbar
.\label{eq:w-diff}
\end{equation}
\end{theorem}

\begin{corollary}
  The composition $U(\mt)\to B\to \oplus B^{\pa}_{\fix}$ is surjective if and only if $\rho^*_{\fix,\fixb}(\pa)-(\chi_{\fix}-\chi_{\fixb})$ is not zero for all $\fix,\fixb$.   In particular, if $\rho_{\fix,\fixb}\neq 0$ for all $\fix,\fixb$, the map $B\to \oplus_{\fix} B^{\pa}_{\fix}$ is surjective for $\pa$ away from finitely many proper subspaces.
\end{corollary}
\begin{proof}
We must have  $\comm_{\fix}-\comm_{\fixb} =
\rho_{\fix,\fixb}' - k_{\fix,\fixb}\hbar $ for some $\rho_{\fix,\fixb}'
\in \Hom(\mt, H^2(\tX) )$ and $k_{\fix,\fixb}\in \mt^*$.

First, we note that $k_{\fix,\fixb}=k_{\fix}-k_{\fixb}$ where
$k_{\fix}=\frac{1}{2} (\anti(\comm_{\fix})-\comm_{\fix}) $.  
 As in the
  proof of \cite[Lem. 5.2]{BLPWgco}, we choose Weyl generators
  $\{x_i,y_i\}$ such that all $x_i$ have positive weight for $\xi$ and
  weight $\chi_i$ for $T$;
  the $y_i$ thus have weight $-\chi_i$.  In these coordinates, we have
  $\comm_{\fix}(t)=\sum \langle t,\chi_i\rangle x_iy_i$.  We have a
  weight preserving
  anti-automorphism on $A^{0}_{fix}$ given by $x_i\mapsto x_i, y_i\mapsto
  y_i,\hbar \mapsto -\hbar$, and the image of $\comm_{\fix}(t)$ under
  this map is\[\sum
  \langle t,\chi_i\rangle y_ix_i=\sum \langle t,\chi_i\rangle
  x_iy_i-\hbar \langle t,\chi_i\rangle.\]  Thus, we have that 
  $k_{\fix}= \nicefrac{\hbar }{2}\sum\chi_i=\chi_{\fix}$, and the second component
  of the 
  formula (\ref{difference-comoment-maps}) is
  correct.
  
We can calculate $\rho_{\fix,\fixb}' $ by considering the effect of tensor product with quantized line bundles.  Whenever we have two classes $\pa,\pa'$ with $\pa-\pa'\in H^2(\tilde{X},\Z)$, we can quantize the unique line bundle with this Chern class to a $\cD{\pa}\operatorname{-}\cD{c'}$-bimodule.  The action
 \[ t\cdot s=\comm_{\fix}(t)s-s \comm_{\fix}(t)\] induces an action of $\ft$ on the sections of this quantization and thus a $T$-equivariant structure on the corresponding line bundle.  This is uniquely characterized by the fact that this action is trivial on the fiber over $\fix$.  Thus, its
  weight on the fiber over $\fixb$ is exactly $\rho_{\fixb,\fix}^*(\pa-\pa')$, and  we can conclude that
  \[\langle \rho_{\fix,\fixb}' (t), \pa-\pa'\rangle =\langle t,
\rho_{\fix,\fixb}^*(\pa-\pa')\rangle .\]
 Since this equation holds for all $\pa - \pa' \in H^2(\tilde{X},\Z)$, 
 we must have that $\rho_{\fix,\fixb}'= \rho_{\fix,\fixb}$ and \eqref{difference-comoment-maps} holds.  
\end{proof}

\begin{example}
  If $\tilde{X}=T^*G/B$, then let $\mathfrak{n}_-\subset \mathfrak{g}$ be the
  nilpotent subalgebra given by the negative $\xi$-weight spaces of
  $\mathfrak{g}$ and $\mt$ be the Cartan subalgebra given by its invariants.   At
  each $T$-fixed points, we have that $T^*\tilde{X}\cong \mathfrak{g}/\mt$.  Thus at each fixed
  point, we have that $T^-_{\fix}\tilde{X}\cong \mathfrak{n}_-$
  as a $T$-representation, and $\chi_{\fix}=-\rho$ for all $\fix$.

  Recall that after identifying $H_2(G/B;\C)\cong \mt$, the maps $\rho_{\fix,\fixb}:\mt\to \mt$ are simply
  $w_{\fix}-w_{\fixb}$ for the corresponding elements of the Weyl group.  In
  particular, this shows that the $B$-algebra will surject to the sum
  $B^{\pa}_{\fix}$ (and thus we will have the full expected number of Verma modules) if and only if $w\pa-w'\pa\neq 0$ for all $w,w'\in W$. That is, if $\pa$ is regular.

  Readers familiar with the theory of category $\cO$ might be confused
  by the absence of ``$\rho$-shifts'' but we have already dealt with
  these in choosing our conventions for $\cD{\pa}$ (these coincide with
  the conventions of \cite{BB}); for example, the usual microlocal differential operators
  are $\cD{\rho}$, so this is consistent with the fact that principal
  block is regular.
\end{example}


\bibliography{./symplectic}
\bibliographystyle{amsalpha}

\end{document}